\documentclass[11pt,reqno]{amsart} 
\usepackage{amsmath, amsthm, mathtools, amssymb}
\usepackage{mathrsfs} 
\usepackage{amscd,stmaryrd}
\usepackage{dirtytalk} 
\usepackage{parskip}
\usepackage{bbm}
\usepackage[english]{babel}
\usepackage[paperheight=11in, 
    paperwidth=8.5in,
    outer=1in,
    inner=1in,
    bottom=1in,
    top=1in]{geometry} 
\usepackage{eurosym}
\usepackage{enumitem}
\usepackage{array}
\usepackage{setspace}
\usepackage{url}
\usepackage[backref=page]{hyperref}
\hypersetup{
	colorlinks,
    citecolor=magenta,
    filecolor=magenta,
    linkcolor=blue,
    urlcolor=black
}
\usepackage{color}
\usepackage[dvipsnames]{xcolor}
\usepackage{textcase}
\usepackage{appendix}
\usepackage{tikz-cd}

\newcommand{\eHK}[1]{e_\text{HK}\left(#1\right)}

\newcommand{\m}{\mathfrak{m}}
\newcommand{\n}{\mathfrak{n}}
\newcommand{\ue}{{\underline{e}}}
\newcommand{\uf}{{\underline{f}}}
\newcommand{\HK}{\text{HK}}
\newcommand{\hs}{\HK_1(S_{p,d})}
\newcommand{\hr}{\HK_1(R_{p,d})}

\newcommand{\qandq}{\quad\text{and}\quad}

\newtheorem{theorem}{Theorem}[section]
\newtheorem{maintheorem}{Theorem}

\newtheorem{definition}[theorem]{Definition}
\newtheorem*{notation}{Notation}
\newtheorem{theoremdefinition}[theorem]{Theorem-Definition}
\newtheorem{corollary}[theorem]{Corollary}
\newtheorem{lemma}[theorem]{Lemma}
\newtheorem{proposition}[theorem]{Proposition}

\newtheorem{conjectureintro}{Conjecture}

\newtheorem{theoremintronumbered}{Theorem}

\newtheoremstyle{example}{10pt}{10pt}{}{}{\bfseries}{.}{.5em}{}
\theoremstyle{example}
\newtheorem{example}[theorem]{Example}

\theoremstyle{remark}
\newtheorem*{remark}{Remark}
\newtheorem*{remarks}{Remarks}
\newtheorem{claim}[theorem]{Claim}

\title{Strong Watanabe–Yoshida conjecture for complete intersections}
\author{Joel Castillo-Rey}
\address{BCAM – Basque Center for Applied Mathematics, Mazarredo 14, 48009 Bilbao, Basque Country – Spain}
\email{jcastillo@bcamath.org}
\date{}

\begin{document}
\begin{abstract}
In this paper, we prove that the $A_1$ singularities are characterised by the minimality of their Hilbert–Kunz multiplicity across singular complete intersections in every positive characteristic. In characteristics 2 and 3, we explicitly compute the Hilbert–Kunz functions of the $A_1$ and $A_2$ singularities.
\end{abstract}
\maketitle

\begin{spacing}{0.8}
\tableofcontents
\end{spacing}

\section*{Introduction}
Let $(R,\mathfrak{m})$ be a local Noetherian ring of positive characteristic $p$ and dimension $d\geq 0$. We denote by $\m^{[p^e]} = (a^{p^e}: a\in \m)$. The function introduced in \cite{k69},
\begin{gather*}
    \text{HK}_e(R):=\operatorname{length}_R{\left(\frac{R}{\m^{[p^e]}}\right)}
\end{gather*}
is called the Hilbert-Kunz function of $R$. Kunz also showed that this function detects singularity of local rings. Monsky \cite{monsky} proved that the limit
\begin{equation*}
    e_{\text{HK}}(R) = \lim_{e\rightarrow \infty} \frac{\HK_e(R)}{p^{ed}}
\end{equation*}
exists, and is at least one. This limit is called the Hilbert–Kunz multiplicity of $R$. Watanabe and Yoshida \cite{wy00} proved that an unmixed local ring with $\eHK{R}=1$ has to be regular. This number need not be an integer, but we know that it does not approach 1 arbitrarily when computed across all singularities of a fixed dimension and characteristic – see \cite{blien} and \cite{aben08}. In \cite{wy05}, Watanabe and Yoshida stated the following conjecture about the minimal value of the Hilbert–Kunz multiplicity over non-regular rings:

\begin{conjectureintro}[\cite{wy05}, \cite{yos}]
\label{watanabe_yoshida}
Let $d\geq 2$. Set $R_{p,d}:=\mathbb{F}_p[[x_0,...,x_d]]/(Q_d),$ where
\begin{equation*}
    Q_d = \begin{cases}x_0x_1+x_2x_3 + \dots + x_{d-1}x_d & \text{if $d$ is odd},\\
    x_0x_1 + x_2x_3 + \dots +x_{d-2}x_{d-1}+x_d^2 & \text{if $d$ is even}.\end{cases}
\end{equation*}
Let $(R,\m,k)$ be an unmixed non-regular local ring of characteristic $p$ and dimension $d$. Then
\begin{enumerate}
    \label{fr: the_inequalities}
    \item $e_{\text{HK}}(R)\geq e_{\text{HK}}(R_{p,d})$,
    \item if $\eHK{R}=\eHK{R_{p,d}}$, then $\widehat{R}\cong R_{p,d}\widehat{\otimes}_{\mathbb{F}_p} k$, i.e.~the completion of $R$ is isomorphic to the coordinate ring of a quadric hypersurface.
    \item $e_{\text{HK}}(R_{p,d})\geq 1+c_d$,
\end{enumerate}
where $c_d$ are the coefficients of the Taylor expansion of $\sec x + \tan x$.
\end{conjectureintro}
Note that by unmixed we mean that $\dim \widehat{R}/\mathfrak{p} = \dim \widehat{R}$ for every $\mathfrak{p}\in \text{Ass}(\widehat{R})$. 

The first two parts are what we refer to as the Strong Watanabe–Yoshida conjecture, since not only do they provide a lower bound for the Hilbert–Kunz multiplicity but also characterize the rings with minimal Hilbert–Kunz multiplicity.

As for the third part, it provides a characteristic-free lower bound for Hilbert–Kunz multiplicity, and can be understood as part of a series of questions concerning the behavior of $\eHK{R_{p,d}}$ that are interesting on their own. Trivedi \cite{tri} showed this for $p$ sufficiently large, and in a recent preprint, Meng \cite{m25} shows that the inequality holds for all $p\geq 3$. The inequality is motivated by Gessel and Monsky's result \cite{gm10} that $\lim_{p\rightarrow\infty}\eHK{R_{p,d}} = 1+c_d,$ and computer experiments. Moreover, Yoshida \cite{yos} conjectured that $\eHK{R_{p,d}}$ is a decreasing sequence on $p$, which Trivedi showed for $p$ large enough in \textit{loc. cit.} 

Different parts of the Watanabe–Yoshida conjecture have been proven, but we still lack a complete answer. The following are the existing results:
\vspace{3em}
\begin{theoremintronumbered}
    With the same notation as in the conjecture above, the following hold.
    \begin{enumerate}
    \item If $d\leq 7$, then $\eHK{R}\geq \eHK{R_{p,d}}\geq 1+c_d$ (\cite[Corollary 2.6, Theorems 3.1 and 4.3]{wy05},\cite[Theorem 5.2]{aben12},\cite[Section 4]{ac24})
    \item If $2\leq d \leq 4$, then $\eHK{R}=\eHK{R_{p,d}}$ if and only if $\widehat{R}\cong R_{p,d}$ (\cite[Corollary 2.6, Theorems 3.1 and 4.3]{wy05}).
    \item If $e(R)\leq 5$, then $\eHK{R}\geq \eHK{R_{p,d}}$, where $e(R)$ is the usual multiplicity of $R$ (\cite[Section 4]{aben12}).
    \item $\eHK{R_{p,d}}\geq 1+c_d$. Also, $\eHK{R_{p,d}}= 1+c_d$ if and only if $d\leq 4$ (\cite[Theorem 8.6]{tri},\cite[Theorem 7.8]{m25})
    \item For $p\gg 0$, $\eHK{R_{p,d}}$ is a decreasing function (\cite[Theorem 8.6]{tri})
    \item For $p>0$, if $R$ is a complete intersection, then $\eHK{R}\geq \eHK{R_{p,d}}$ (\cite[Theorem 4.6]{enshi})
    \end{enumerate}
\end{theoremintronumbered}

We want to focus on the result for complete intersections. In 2005, Enescu and Shimomoto proved the first inequality of the conjecture for complete intersections in characteristic $p>2$, and provided a stronger result for hypersurfaces of characteristic $p>3$:
\begin{theoremintronumbered}[Enescu–Shimomoto, \cite{enshi}, Theorems 3.4 and 4.6]
    \label{th: enescu_shimomoto_intro}
    Let $(R,\m,k)$ be a non-regular complete intersection of characteristic $p>2$ and dimension $d\geq 1$. Then, $\eHK{R}\geq \eHK{R_{p,d}}.$ Moreover, if $p>3$, $R$ is a hypersurface ring, then either $\widehat{R}\cong R_{p,d}\widehat{\otimes}_{\mathbb{F}_p} k$, or
    \begin{gather*}
    \eHK{R}\geq \eHK{S_{p,d}}\geq \eHK{R_{p,d}},
    \end{gather*}
    where $S_{p,d} = \mathbb{F}_p[[x_0,\dots,x_d]]/(x_0^2+\dots+x_{d-1}^2+x_d^3)$.
\end{theoremintronumbered}
The goal of this work is to extend and refine this result to settle the Strong Watanabe–Yoshida conjecture fully across complete intersections in every positive characteristic, which requires tackling the following problems: 

\begin{enumerate}
\item Proving that $\eHK{S_{p,d}}>\eHK{R_{p,d}}.$ In other words, showing that the Hilbert–Kunz multiplicity distinguishes the $A_1$ and $A_2$ singularities. (Corollary \ref{cr: A_2_A_1_ineq_large_char} and Theorem \ref{th: A_1_A_2_ineq_char_3})
\item Generalizing the inequality $\eHK{R}\geq \eHK{S_{p,d}}$ to singular complete intersections that are not isomorphic to an $A_1$ singularity (Theorem \ref{th: reduction_to_hypersurface})
\item Extending all the results to characteristic 2, with the definition of $A_2$ singularity appropriate for this case. (Definition \ref{df: Qd}, Theorem \ref{th: enescu_shimomoto_char_2} and Theorem \ref{th: A_1_A_2_ineq_char_2})
\end{enumerate}
Settling these problems yields the main theorem of this work:
\begin{maintheorem}[Corollary \ref{maintheoremA}]
    Let $(R,\m,k)$ be a non-regular complete intersection of characteristic $p>0$. Then, either $\widehat{R}\cong R_{p,d}\widehat{\otimes}_{\mathbb{F}_p} k$, or
    \begin{gather*}
    \eHK{R}\geq \eHK{S_{p,d}}>\eHK{R_{p,d}},
    \end{gather*}
for every $p>0$ and $d\geq 2$. 
In particular, the Strong Watanabe–Yoshida conjecture holds for complete intersections.
\end{maintheorem}

The cases of characteristic $p>3$ and $p = 2,3$ are approached very differently since in the case of $p=2,3$ we can explicitly compute the Hilbert–Kunz function and multiplicity of the $A_1$ and $A_2$ singularities:
\begin{maintheorem}[Theorems \ref{th: HK_function_A_1_char_2}, \ref{th: HK_function_A_2_char_2}, \ref{th: eHK_A_1_and_A_2_in_char_3}]
    For $d\geq 1$,
        \begin{gather*}
            \eHK{R_{2,d}} = \begin{dcases}
                \frac{2^m}{2^m-1} & (d=2m-1) \\
                \frac{2^m+1}{2^m} & (d=2m)\\
            \end{dcases}\qandq\eHK{S_{2,d}} = \begin{dcases}
                \frac{2^{m-1}+1}{2^{m-1}} & (d=2m-1) \\
                \frac{2^{m+1}+1}{2^{m+1}-1} & (d=2m) \\
            \end{dcases}
        \end{gather*}
        and 
        \begin{gather*}
        \eHK{R_{3,d}} = 1+\left(\frac{2}{3}\right)^{d-1} \qandq \eHK{S_{3,d}} = 1+\frac{3\cdot 2^d}{3^{d+1}-2^d+(-1)^d}.
        \end{gather*}
    \end{maintheorem}

    Theorem \ref{th: HK_function_A_1_char_2} confirms Yoshida's conjecture on the formula of the Hilbert–Kunz function of $\eHK{R_{2,d}}$. He already computed it in the odd-dimensional case with a trick that we will exploit here to quickly compute $\eHK{S_{p,d}}$ in characteristic 3 (see §\ref{subsec:A2_in_char_3}). These computations happen to shed some light on Conjecture 2.7 in \cite{jnskw}, which proposed a better lower bound than the one claimed by Watanabe and Yoshida's conjecture, but Theorem \ref{cor: better_char_free_lower_bound} shows that it only holds in low dimension.

    Finally, after slight refinements of the previous results on the conjecture, it is also possible to settle its stronger form across low-dimensional or low-multiplicity singularities, as well as extend them to the previously unaddressed characteristic 2 case:

    \begin{maintheorem}[Theorem \ref{th: watanabe_yoshida_low_dimension}, \ref{cor: watanabe_yoshida_low_multiplicity}]
        The Strong Watanabe–Yoshida conjecture holds in dimension $d\leq 6$ or multiplicity $e(R)\leq 5$, in characteristic $p>0$. Also if $d=7$ and $p>2$.
    \end{maintheorem}

Let us now lay out the structure of this paper. First, Section \ref{sec:han_monsky_ring} provides some preliminaries on the Han–Monsky ring that will be used in the following two sections, and introduces a new technique that we will exploit in the characteristic 2 case. Sections \ref{sec:strongwatanabeyoshida} and \ref{sec:HKmultiplicity_A1_and_A2_char_2_3} are devoted to showing that the $A_1$ and $A_2$ singularities indeed have different Hilbert–Kunz multiplicity. The Hilbert–Kunz functions of the $A_1$ and $A_2$ singularities in characteristic $2$ and $3$ are explicitly computed in the latter. In Section \ref{sec: wyc_revisited}, we finally settle the Strong Watanabe–Yoshida conjecture for complete intersections, leveraging on the inequalities proven in the previous sections, refining Enescu and Shimomoto's reduction from complete intersections to hypersurfaces, and extending their results to characteristics 2 and 3. Finally, in Section \ref{sec:revisiting_literature}, we will revisit the literature on low-dimensional and low-multiplicity cases of the conjecture.

\textbf{Acknowledgments.} I wish to thank Ilya Smirnov for introducing me to the subject and guiding me through the project. I also wish to thank Kriti Goel for useful discussions and Devlin Mallory for answering my questions and helping with the manuscript.

The author was supported by the grants SEV-2023-2026 from the Spanish Ministry of Science and Innovation, RYC2020-028976-I funded by MICIU/AEI/10.13039/501100011033, and EI ESF "ESF Investing in your future".

\section{The Han–Monsky ring and the Han–Monsky algorithm}\label{sec:han_monsky_ring}
The primary tool we will employ for the main result of this paper is Han and Monsky's representation ring, also called the Han–Monsky ring. This ring is the Grothendieck group of finitely generated $k[T]$-modules where the action of $T$ is nilpotent. It was introduced in \cite{hanmonsky} to compute the Hilbert–Kunz function of diagonal hypersurfaces. In §\ref{subsec: han_monsky_definitions}, we introduce the basic definitions. In §\ref{subsec: hmring_char_2}, we focus on the characteristic 2 case and introduce a new tool for studying the Han–Monsky ring, which proves especially useful in this low characteristic case. Finally, in §\ref{subsec: hmalgorithm}, we explicitly describe the Han–Monsky algorithm, the method for computing the Hilbert–Kunz function of a diagonal hypersurface.

\subsection{Definitions}\label{subsec: han_monsky_definitions}
This exposition of the Han–Monsky ring is strongly based on Yoshida's unpublished notes \cite{yos}.
\begin{definition}[$k$-object]
Let $k$ be a field. A $k$-object is a finitely generated $k[T]$-module $M$ such that the action of $T$ over $M$ is nilpotent.
\end{definition}
\begin{notation}
    \label{not: notation_k_objects_as_pairs}
We often obtain $k$-objects from taking a zero-dimensional finitely generated $k$-algebra $R$ and defining the $T$-action by multiplication by an element $f\in R$. We will often use the notation $(R,f)$ for such $k$-objects.
\end{notation}

\begin{proposition}[Definition 1.3, \cite{hanmonsky}]
    \label{pr: product_and_sum_in_rep_ring}
Let $M,N$ be two $k$-objects. The $k$-vector spaces $M\otimes_k N$ and $M\oplus N$ can be regarded as $k$-objects if we define
\begin{align*}
    T\cdot (m\oplus n) &:= Tm\oplus Tn\\
    T\cdot(m\otimes_k n) &:= Tm\otimes_k n+m\otimes_k Tn.
\end{align*}
\end{proposition}

\begin{example}
    \label{ex: example_diagonal_hypersurface_k_object}
    Take the $k$-object $k[x,y]/(x^a,y^b)$ with $T$-action $x^c+y^d$. Then,
    \begin{gather*}
        \frac{k[x,y]}{(x^a,y^b)} = \frac{k[x]}{(x^a)}\otimes \frac{k[y]}{(y^b)},
    \end{gather*}
    where the $k$-objects on the right side have $T$-action $x^c$ and $y^d$, respectively. Generally, any $T$-action given by a diagonal equation will split this way.
\end{example}
We introduce the following equivalence relation over the pairs of $k$-objects: Let $(M_1,N_1)$ and $(M_2,N_2)$ be two pairs of $k$-objects. Then,
\begin{equation*}
    (M_1,N_1)\sim (M_2,N_2)\Leftrightarrow M_1 \oplus N_2 \cong M_2\oplus N_1, \text{ as $k[T]$-modules.}
\end{equation*}
The formal difference $M-N$ denotes the equivalence class containing $(M,N)$. In other words, this is the structure of the Grothendieck ring on the free monoid generated by the equivalence classes of $k$-objects.
\begin{definition}[Han–Monsky ring]
Let
\begin{equation*}
    \Gamma_k = \lbrace M- N\mid M \text{ and } N \text{ are $k$-objects}\rbrace,
\end{equation*}
and define $+$ and $\cdot$ in $\Gamma_k$ as follows:
\begin{align*}
    (M_1-N_1)+(M_2-N_2) &= (M_1\oplus M_2) - (N_1\oplus N_2),\\
    (M_1-N_1)\cdot (M_2-N_2) &= \left[(M_1\otimes M_2)\oplus (N_1\otimes N_2)\right] - \left[(M_1\otimes N_2)\oplus (M_2\otimes N_1)\right].
\end{align*}
This way we have a commutative ring with unity $k-0$, called the Han–Monsky ring over $k$.
\end{definition}
Usually, there is no ambiguity in denoting $\Gamma = \Gamma_k$, because $\Gamma$ only depends on the characteristic of the field, which is usually clear from the context.

\begin{definition}
Let $\delta_i := k[T]/(T^i),$ for every $i\geq 1$. We call these elements the $\delta$-basis.
\end{definition}

The fact that the $\delta$-basis indeed forms a $\mathbb{Z}$-basis is a rewording of the structure theorem for finitely generated modules over a PID:

\begin{corollary}
    \label{cor: delta_basis_is_a_basis}
The abelian group $\Gamma$ is a free $\mathbb{Z}$-module where $\delta_i$ is a free $\mathbb{Z}$-basis.
\end{corollary}
\begin{remark}
Let $k[x]/(x^a)$ be a $k$-object with $T$-action $x^c$. Let $a = ck+r$. Then, one can decompose $k[x]/(x^a)$ as $(c-r)\delta_k+ r\delta_{k+1}$. Looking back at Example \ref{ex: example_diagonal_hypersurface_k_object}, if $k[y]/y^b$ has $T$-action $x^d$ and $b = dl+s$, then $k[x,y]/(x^a,y^b)$ with $T$-action $x^c+y^d$ decomposes in the Han–Monsky ring as 
\begin{gather*}
    \left((b-r)\delta_k+r\delta_{k+1}\right)\left((d-s)\delta_l+s\delta_{l+1}\right)
\end{gather*}
Therefore, understanding the products $\delta_i\delta_j$ is crucial to get to a decomposition over the $\delta$-basis.
\end{remark}

To understand the multiplication table of $\Gamma$, a different $\mathbb{Z}$-basis is introduced, which has a simpler behavior.
\begin{definition}[$\lambda$-basis]
    \label{def: lambda}
    Let $\lambda_i := (-1)^i(\delta_{i+1}-\delta_i), i\geq 0$, noting that $\delta_0 = 0$ in $\Gamma$.
\end{definition}
\begin{remark}
    Note that $\delta_i = \lambda_0 - \lambda_1 + \dots + (-1)^i\lambda_{i-1},\text{ for } i\geq 1.$
\end{remark}
To finish, we define a $\mathbb{Z}$-linear function over $\Gamma$. 
\begin{definition}
    \label{def: alpha_D_k_and_ell_k}
    Let $M-0 = \bigoplus \delta_i^{\oplus a_i}$, and define
\begin{equation*}
    \alpha(M-0) := \dim_k M/TM = \sum_{i=1}^\infty a_i
\end{equation*}
Extending by linearity to all pairs $M-N$, we obtain a $\mathbb{Z}$-linear function $\alpha:\Gamma\rightarrow \mathbb{Z}$. For $s,a_i,b_i\in \mathbb{Z}$ with $s\geq 1$, $a_i\geq 1$, and $b_i\geq 0$, we denote
\begin{gather*}
D_k(a_1,\dots,a_s) := \alpha(\prod_{i=1}^s\delta_{a_i}) \hspace{4em} \ell_k(b_1,\dots,b_s) := \alpha(\prod_{i=1}^s\lambda_{a_i}).
\end{gather*}
\end{definition}
\begin{remark}
Note that $\alpha(\delta_j) = 1, \forall j\geq 1$, and
\begin{align*} 
\alpha(\lambda_i) = \begin{cases} 1 & (i=0),\\ 0 & (i\geq 1).\end{cases}
\end{align*}
    Now, let $R = k[[x_1,\dots,x_s]]/(f)$ be a characteristic $p>0$ hypersurface. Then 
    \begin{gather*}
        \HK_e(R) = \dim_k\left(\frac{k[[x_1,\dots,x_s]]}{(f)+\mathfrak{m}^{[p^e]}}\right).
    \end{gather*}
    Observe that we can regard $\mu_e := k[[x_1,\dots,x_s]]/\mathfrak{m}^{[p^e]}$ as a $k$-object with $T$-action $f$, so that $\HK_e(R) = \dim_k(\mu_e/T\mu_e).$ If $\mu_e = \sum a_i\delta_i$ with the given $k$-object structure, then $\HK_e(R) = \sum a_i$.
\end{remark}
\subsection{The Han–Monsky ring in characteristic 2}\label{subsec: hmring_char_2}
The Han–Monsky ring $\Gamma$ can be filtered by subrings $\Lambda_1\subset \dots \subset \Lambda_e\subset ...$ whose arithmetic is easier to understand:
\begin{theoremdefinition}
    [Theorem 3.2, \cite{hanmonsky}]
    \label{lambdadef}
        The additive subgroup of $\Gamma$ generated by the $\lambda_i$ for $0\leq i \leq p^e-1$ is closed under multiplication. We call this ring $\Lambda_e$.
\end{theoremdefinition}
\begin{theorem}[Theorem 2.5, \cite{hanmonsky}]
    \label{lm: structure_Lambda_1}
    Let $i\leq j\le p-1$. Then $\lambda_i\lambda_j = \sum_{j-i}^{\min{i+j,2p-2-i-j}} \lambda_k$.
\end{theorem}
\begin{table}[h]
    \label{tab: multiplication_table_lambda_1_char_2_and_char_3}
\centering
\begin{tabular}{c|cc}
    $\cdot$ & $\lambda_0$ & $\lambda_1$ \\ \hline
    $\lambda_0$ & $\lambda_0$ & $\lambda_1$ \\
    $\lambda_1$ & $\cdot$ & $\lambda_0$
\end{tabular}\hspace{2em}
\begin{tabular}{c|ccc}
$\cdot$ & $\lambda_0$ & $\lambda_1$ & $\lambda_2$ \\ \hline
$\lambda_0$ & $\lambda_0$ & $\lambda_1$ & $\lambda_2$ \\
$\lambda_1$ & $\cdot$ & $\lambda_0+\lambda_1+\lambda_2$ & $\lambda_1$ \\
$\lambda_2$ & $\cdot$ & $\cdot$ & $\lambda_0$
\end{tabular}
\caption{Multiplication table for the $\lambda$-basis of $\Lambda_1$ in characteristic $2$ and $3$, respectively.}
\end{table}

\begin{proposition}
\label{pr: delta_p_squared}
$\delta_p^2 = p\delta_p$.
\end{proposition}
\begin{proof}
First part of Lemma 3.3 in \cite{hanmonsky}.
\end{proof}

    \begin{lemma}[Lemma 3.4, \cite{hanmonsky}]
        \label{lm: lambda_p_adic_split}
        Let $q = p^n$, $n\geq 1$. Take $0\leq i \leq q-1.$ Then, in $\Gamma$,
        \begin{enumerate}[label=(\alph*)]
        \item $\lambda_i\lambda_{qj} =\lambda_{qj+i}$; \label{lm: p_adic_split}
        \item $\lambda_i\lambda_{qj-1} = \lambda_{qj-i-1}$. \label{lm: p_adic_opposite}
        \end{enumerate}
    \end{lemma}
    \begin{corollary}
    \label{cr: p_adic_expansion_rep_ring}
    Let $0\leq a< p^{n+1},$ and $a = a_0+a_1p+a_2p^2+\dots+a_np^n$ be its expansion in base $p$, $0\leq a_i<p$. Then, we have the following decomposition:
    \begin{gather*}
        \lambda_a = \prod_{k\geq 0}^n\lambda_{a_kp^k}.
    \end{gather*}
    \end{corollary}
    \begin{proof}
        We proceed by induction on $n$. For $n=0$, there is nothing to do. Now, observe that
        \begin{gather*}
            a_0 + pa_1 +\dots + p^na_n = (a_0 + pa_1+\dots+p^{n-1}a_{n-1})+p^na_n,
        \end{gather*}
    where $a_0+pa_1+\dots+p^{n-1}a_{n-1}<p^n.$ Thus, the lemma above gives
    \begin{gather*}
        \lambda_{a_0+pa_1+\dots+p^{n-1}a_{n-1}+p^na_n} = \lambda_{a_0+pa_1+\dots+p^{n-1}a_{n-1}}\lambda_{p^na_n}.
    \end{gather*}
    \end{proof}
This result implies that, to understand the multiplication table of $\Gamma$, it is enough to understand the multiplication table of the elements $\lbrace \lambda_{qi}\rbrace_{0\leq i\leq q-1}$ for each $q = p^e$. This becomes especially simple in characteristic 2:
\begin{lemma}[Theorem 3.7, \cite{hanmonsky}]
    \label{cr: lambdaqsquared}
        Given $q=2^n,$ then $\lambda_{q}^2 = \lambda_0$.
\end{lemma}

It is useful to allow rational coefficients in the Han–Monsky ring:
\begin{definition}[Definition 4.8, \cite{hanmonsky}]
\label{df: representation_ring_rational_coefficients}
    Let $a\geq 0$ be a real number. In $\mathbb{Q}\otimes_\mathbb{Z} \Gamma$, set 
    \begin{gather*}
        \delta_a := (1-z)\delta_r + z\delta_{r+1}, \text{ where } a=r+z, r\in \mathbb{Z}_{\geq 0}, z\in [0,1).
    \end{gather*}
    We extend $\alpha$ to $\mathbb{Q}\otimes_\mathbb{Z} \Gamma$ by $\mathbb{Q}$-linearity, and denote $D(a_i,\dots,a_i) := \alpha(\prod \delta_{a_i})$. Note that we are assuming $\delta_0 = 0$.
    \end{definition}

    \begin{remark}
        In the same way, we also allow rational coefficients in the subrings $\Lambda_e$ defining $\Lambda_e\otimes_\mathbb{Z} \mathbb{Q}$.
    \end{remark}
    
    We introduce a new basis that happens to be very useful for computations, and remarkably simple in characteristic 2. This basis is a set of orthogonal idempotent elements, that generate $\Lambda_e\otimes_\mathbb{Z} \mathbb{Q}$. It will be the main tool to compute the Hilbert–Kunz functions of the $A_1$ and $A_2$ singularities in characteristic 2, later in Section \ref{sec:HKmultiplicity_A1_and_A2_char_2_3}.
\begin{definition}[$\sigma$-basis]
Let $e \geq 1$ and $0 \leq k \leq 2^e - 1$. Write $k$ in binary as $k = k_0 + 2k_1 + \dots + 2^{e-1}k_{e-1}$, where $k_i \in \{0, 1\}$. We define the $\sigma$-basis as:
\begin{equation*}
\sigma_{k,e} := \frac{1}{2^e} \prod_{j=0}^{e-1} (\lambda_0 + (-1)^{k_j} \lambda_{2^j}) \in \Lambda_e\otimes_\mathbb{Z} \mathbb{Q}
\end{equation*}
where $\lambda_i$ are the elements of the $\lambda$-basis.
\end{definition}
\begin{remark}
    Soon after the first version of this preprint was uploaded, Nick Cox-Steib let me know that he had been working on this basis in the case of characteristic $\neq 2$. It is worth mentioning that their coordinates over the $\lambda$-basis seem to involve irrational coefficients.
\end{remark}
\begin{proposition}
\label{pr: sigmaproperties}
Let $e\geq 1$. The $\sigma$-elements have the following properties:
    \begin{enumerate}[label = (\alph*)]
        \item (Idempotency and orthogonality) $\sigma_{k,e}^2=\sigma_{k,e}$ and $\sigma_{k,e}\sigma_{l,e} = 0$ for $k\neq l$. \label{pr: idem_orth}
        \item If $\eta = \sum_{k=0}^{2^e-1} a_k\sigma_{k,e}$, then $\sigma_{k,e}\eta = a_k\sigma_{k,e}$. \label{pr: extracting_coordinates_orthogonality}
        \item (Interaction between $\sigma$-bases) Let $e\leq e'$,
        \begin{gather*}
            \sigma_{l,e'}\sigma_{k,e} = \begin{cases}
                \sigma_{l,e'} & \text{ if } l=k \text{ or } |l-k| \geq 2^{e},\\
                0 & \text{otherwise.}
                \end{cases}
        \end{gather*}
        In particular, if $l\leq 2^e$, then $\sigma_{l,e'}\sigma_{k,e} = \sigma_{l,e'}$ if $l = k$, and 0 otherwise. \label{pr: inter_orthogonality}
        \item (Recursive split) $\sigma_{k,e} = \sigma_{k,e+1}+\sigma_{2^e+k,e+1}$. \label{pr: recursive_split}
        \item Let $n = n_0+2n_1+\dots+2^{e-1}n_{e-1}$, $n_i = 0,1$. Then, $\sigma_{k,e} = \frac{1}{2^e}\sum_{i=0}^{2^e-1} (-1)^{\sum i_jk_j} \lambda_{i}$. \label{pr: sigma_in_terms_of_lambda}
        \item $\sigma_{1,e} = \frac{1}{2^e}\delta_{2^e} = \frac{1}{2^e}\left(\sum_{j=0}^{2^e-1} (-1)^j\lambda_{j}\right)$. \label{pr: sigma_1_delta_and_lambda}
        \item $\alpha(\sigma_{k,e}) = \frac{1}{2^e}$ for every $k$. \label{pr: alpha_of_sigma}
    \end{enumerate}
\end{proposition}
\begin{proof}
    Parts \ref{pr: sigma_1_delta_and_lambda} and \ref{pr: alpha_of_sigma} follow directly from \ref{pr: sigma_in_terms_of_lambda}.
\begin{enumerate}[label = (\alph*)]
\item From Corollary \ref{cr: lambdaqsquared} and the fact that $\lambda_0$ is the unit of the Han–Monsky ring, one checks that $(\lambda_0+\lambda_{2^n})^2=2(\lambda_0+\lambda_{2^n})$ and $(\lambda_0+\lambda_{2^n})(\lambda_0-\lambda_{2^n}) = 0.$ Applying the definition, the claim follows.
\item Follows directly from part \ref{pr: idem_orth}.
\item Both by the fact that $k<2^n$ and $|l-k|\geq 2^n$ we conclude that $l_j=k_j$ for $j=0,\dots,n-1$. Thus,
\begin{gather*}
    \left(\prod^{n-1}_{j=0} (\lambda_0+(-1)^{k_j}\lambda_{2^j})\right)\cdot\left(\prod^{m-1}_{j=0}(\lambda_0+(-1)^{l_j}\lambda_{2^j})\right) = \\
    \left(\prod^{n-1}_{j=0} (\lambda_0+(-1)^{l_j}\lambda_{2^j})^2\right)\cdot\left(\prod^{m-1}_{j=n}(\lambda_0+(-1)^{l_j}\lambda_{2^j})\right)= 2^n\prod^{m-1}_{j=0}(\lambda_0+(-1)^{l_j}\lambda_{2^j}).
\end{gather*}
\item By definition, $\sigma_{k,e+1} + \sigma_{2^{e}+k,e+1} = \frac{1}{2}\left(\sigma_{k,e}\cdot (\lambda_0+\lambda_{2^e})+\sigma_{k,e}\cdot (\lambda_0-\lambda_{2^e})\right) = \sigma_{k,e}.$
\item By induction on $e$. The base case is clear since $\sigma_{0,1} = \frac{1}{2}(\lambda_0+\lambda_1)$ and $\sigma_{1,1} = \frac{1}{2}(\lambda_0-\lambda_1)$. Now, by the induction hypothesis,
\begin{gather*}
\prod_{j=0}^{e-1} (\lambda_0+(-1)^{k_j}\lambda_{2^j}) = (\lambda_0+(-1)^{k_{e-1}}\lambda_{2^{e-1}})\sum_{i=0}^{2^{e-1}-1} (-1)^{\sum_{j=0}^{e-2} i_jk_j}\lambda_i,
\end{gather*}
which yields the formula.
\begin{gather*}
    \sigma_{1,e} := \frac{1}{2^e}(\lambda_0-\lambda_1)\prod_{j=1}^{e-1}(\lambda_0+\lambda_{2^j}) = \frac{1}{2^e}\left(\lambda_0-\lambda_1+\dots+(-1)^{e-1}\lambda_{2^e-1}\right) = \frac{1}{2^e}\delta_{2^e}.
\end{gather*}

\end{enumerate}
\end{proof}
\begin{remark}
    Since the $\sigma$-elements form a $\mathbb{Q}$-basis of $\Lambda_e\otimes \mathbb{Q}$, the statement of \ref{pr: sigmaproperties}\ref{pr: idem_orth} represents the fact that $\Lambda_e\otimes \mathbb{Q}$ is isomorphic as a $\mathbb{Q}$-algebra to the direct product of $p^e$ copies of $\mathbb{Q}$, where the $\sigma$-elements would be the canonical basis of this product.
\end{remark}

Now that we have settled these first properties, we provide two technical lemmas that we will use for computations in later sections. They establish the interaction of the $\sigma$-basis with some elements of the $\lambda$-basis.
\begin{lemma}
    \label{lm: lambda_sigma_product_e}
    Let $0\leq k\leq 2^e-1$, $e\geq 1$, and denote $k = k_0+2k_1+...+2^{e-1}k_{e-1}$ where $0\leq k_j\leq 1$ integers. Then, $\sigma_{k,e}\lambda_{2^e-1} = (-1)^{\sum_{j=0}^{e-1} k_j}\sigma_{k,e}.$
\end{lemma}
\begin{proof}
By \ref{pr: sigmaproperties}\ref{pr: sigma_in_terms_of_lambda}, recall that $\sigma_{k,e} = \sum_{i=0}^{2^e-1} (-1)^{\sum i_jk_j} \lambda_i$. Observe that $2^e-1 = 1+2+\dots+2^{e-1}$, so $2^e-i-1 = (1-i_0)+(1-i_1)2+(1-i_2)4+\dots+(1-i_{e-1})2^{e-1}$. Therefore,
\begin{gather*}
    \lambda_{2^e-1}\sigma_{k,e} = \sum_{i=0}^{2^e-1} (-1)^{\sum i_jk_j} \lambda_{2^e-i-1} = \sum_{i=0}^{2^e-1} (-1)^{\sum (1-i_j)k_j} \lambda_i = (-1)^{\sum_{j=0}^{e-1} k_j}\sigma_{k,e}.
\end{gather*}
\end{proof}
\begin{lemma}
    \label{lm: lambda_sigma_product_e+1}
    Let $0\leq k\leq 2^{e+1}-1$, $e\geq 1$. In the same notation as the previous lemma, we have that $\sigma_{k,e+1}\lambda_{2^e-1} = (-1)^{\sum_{j=0}^{e-1} k_j}\sigma_{k,e+1}.$
\end{lemma}
\begin{proof}
    Write $k' = k$ if $k\leq 2^e-1$ and $k-2^e$ if $2^e\leq k \leq 2^{e+1}-1$. Then $\sigma_{k,e+1} = \sigma_{k',e}(\lambda_0+(-1)^{k_e}\lambda_{2^e})$, and we apply Lemma \ref{lm: lambda_sigma_product_e}:
    \begin{align*}
        \lambda_{2^e-1}\sigma_{k,e+1} &= \lambda_{2^e-1}\sigma_{k',e}(\lambda_0+(-1)^{k_e}\lambda_{2^e})\\
        &= (-1)^{\sum_{j=0}^{e-1} k_j} \sigma_{k',e}(\lambda_0+(-1)^{k_e}\lambda_{2^e}) = (-1)^{\sum_{j=0}^{e-1} k_j}\sigma_{k,e+1}.
    \end{align*}
\end{proof}
\begin{remark}
    Thus, for example, $\sigma_{2^e,2^{e+1}}\lambda_{2^e-1} = \sigma_{2^e+1,2^{e+1}},$ and $\sigma_{2^e,2^{e+1}}\lambda_{2^e-1} = -\sigma_{2^e+1,2^{e+1}}$.
    \end{remark}
    \subsection{The Han–Monsky algorithm}\label{subsec: hmalgorithm}
    As we mentioned at the beginning of the section, the Han–Monsky ring is particularly powerful for determining the Hilbert–Kunz function of diagonal hypersurfaces. In this section, we will recap some of the results in this direction, as presented in \cite{hanmonsky}. Let $R$ denote the ring of power series at the origin of a diagonal hypersurface: 
    \begin{gather*}
        R := \frac{k[[x_0,\dots,x_d]]}{(x_0^{e_0}+\dots+x_d^{e_d})}.
    \end{gather*}
    Note that none of the following results depend on the field $k$ but only on its characteristic, so we can assume that $k=\mathbb{F}_p$.

    \begin{lemma}[Lemma 5.6, \cite{hanmonsky}]
    \label{lm: lemmarepring}
        Let $q=p^e$. The $k$-object $k[[x_0,\dots,x_d]]/(x_0^{q},\dots,x_d^{q})$ with $T$-action $x_0^{e_0}+\dots+x_d^{e_d}$ is $\prod_{i=0}^d \delta_{q/e_i}$ in $\Gamma'$. Therefore,
        \begin{gather*}
            \HK_e(R) = e_0\dots e_dD\left(\frac{q}{e_0},\dots,\frac{q}{e_d}\right).
        \end{gather*}
    \end{lemma}
    
    This result expresses that we can understand the Hilbert–Kunz function of diagonal hypersurfaces by studying the function $D(a_0,\dots,a_s)$.
    
    \begin{definition}
    \label{df: vectors}
        Let $\beta$ be a vector of positive real numbers and $p$ be a prime. Let $\mu$ be a positive integer that we call the period. We obtain the following data associated with these $\beta, \mu$, and $p$.
        \begin{enumerate}
            \item Let $r\in \mathbb{Z}_{\geq 0}^{d+1}$ and $z\in [0,1)^{d+1}$ such that $\beta = r + z$.
            \item Let $R\in \mathbb{Z}_{\geq 0}^{d+1}$ and $z'\in [0,1)^{d+1}$ such that $p^{\mu}\beta = R + z'$.
            \item Let $t := p^{\mu}z - z' \in \mathbb{Z}_{\geq 0}^{d+1}$.
        \end{enumerate}
        We denote the coordinates as $r = (r_0,\dots,r_d)$.
    
        Let $v,w\in \mathbb{R}^{d+1}$ be arbitrary vectors. We will say that $v$ is \textbf{equivalent} to $w$, denoted $v\sim w$, if we can obtain it from $w$ by replacing $v_i$ by $1-v_i$ for an even number of indices $i$ (so in particular, if they are equal to each other). We also introduce the notation $v^*:= (1-v_0,v_1,\dots,v_d)$.
        
        Finally, we define the function $\ell^\sharp$:
        \begin{gather*}
            \ell^\sharp = \begin{cases} \ell_k(t) & \text{if } \sum r_i \text{ even,}\\ \ell_k(p^{\mu}-1-t_0,t_1,\dots,t_d) & \text{if } \sum r_i \text{ odd,}\end{cases}
        \end{gather*}
        for $\ell_k$ as in Definition \ref{def: alpha_D_k_and_ell_k}. We will refer to the vectors $r,z,R,z',t$ and the scalar $\ell^\sharp$ as the data associated with $\beta,\mu$ and $p$.
    \end{definition}
    
    \begin{lemma}[Theorem 5.3 and remark, \cite{hanmonsky}]
    \label{lm: th53}
        Let $\beta$ be a vector of positive real numbers, and $\mu$ a positive integer, and let $r,z,R,z'$, and $t$ be the data associated with $\beta,\mu$, and $p$. Suppose that either one of these two sets of conditions is satisfied by the data:
        \begin{enumerate}
            \item $z'\sim z$ and $\sum R_i \equiv \sum r_i \mod{2}$, or
            \item $z'\sim z^*$ and $\sum R_i \not\equiv \sum r_i \mod{2}$.
        \end{enumerate}  Then, there is a rational number $C$ and rational numbers $\Delta_0,\Delta_1,\Delta_2,\dots$ such that 
        \begin{gather*}
            D(p^e\beta) = Cp^{de} + \Delta_e,
        \end{gather*} and $\Delta_{e+\mu}=\ell^\sharp \Delta_e$ for all $e\geq 0$, where $\ell^\sharp$ was defined in \ref{df: vectors}.
    \end{lemma}
    
    The following theorem shows how Lemma \ref{lm: th53} can be used to compute the Hilbert–Kunz multiplicity (and function) of a diagonal hypersurface.

    \begin{theorem}[Han–Monsky algorithm, \cite{hanmonsky}, \cite{yos}]
        \label{th: han_monsky_algorithm}
        Let $p>0$, and define
        \begin{gather*}
        R = \frac{\mathbb{F}_p[[x_0,\dots,x_d]]}{(x_0^{e_0}+\dots+x_d^{e_d})}.
        \end{gather*}
        Then, there exist integers $n_0\geq 0$ and $\mu\geq 1$ such that the data associated to $\beta = p^{n_0}(1/e_0,\dots,1/e_d)$ and $\mu$ verify the first set of conditions in Lemma \ref{lm: th53}. In that case, then
        \begin{gather*}
            \eHK{R} = \frac{\text{HK}_{n_0+\mu}(R)-\HK_{n_0}(R)\ell^\sharp}{p^{dn_0}-\ell^\sharp},
        \end{gather*}
        where $\ell^\sharp$ is as in Definition \ref{def: alpha_D_k_and_ell_k}, for these $\beta$ and $\mu$.
    \end{theorem}
        \begin{proof}
            By Lemma \ref{lm: lemmarepring}, we know that $\HK_e(S_{p,d}) = e_0\dots e_d D(q\beta)$ for $\tilde{\beta} = \left(1/e_0,\dots,1/e_d\right)$. It can be seen that, if $e_i$ is divisible by $p$ for some $i$, then there is no $\mu$ such that $\tilde{\beta}$ and $\mu$ verifies none of the conditions in \ref{lm: th53}, since $z'$ will always have a zero coordinate. Hence, let $n_0\geq 0$ be the smallest non-negative integer such that none of the denominators in $p^{n_0}\tilde{\beta}$ are divisible by $p$, and define $\beta := p^{n_0}\tilde{\beta}$.

            Now, one can check that there exists $\mu\geq 1$ such that all the data associated to $\beta$ and $\mu$ verify the first set of conditions in Lemma \ref{lm: th53}: indeed, let $\mu\geq 1$ be any $\mu$ such that $(p^{\mu+n_0}-1)p^{n_0}\beta \in \mathbb{Z}^{d+1}$, and possibly changing $\mu$ by $2\mu$, we can also assume that $(p^\mu-1)\sum \beta_i \equiv 0 \bmod{2}$. The claim follows observing that $\sum R_i-\sum r_i = (p^\mu-1)\sum \beta_i \equiv 0 \bmod{2}$.
            
            Thus, applying Lemma \ref{lm: th53}, there exists $C$, $\ell^\sharp$, and $\Delta_e$ such that
            \begin{gather*}
                2^d3D(p^e\beta) = 2^d3Cp^{de}+2^d3\Delta_e,
            \end{gather*}
            where $\Delta_{e+\mu} = \ell^\sharp \Delta_e$ for all $e\geq 0$. Note that $\eHK{R} = e_0\dots e_dC$ and let $c_1 = -e_0\dots e_d\Delta_{n_0}$, so
            \begin{gather*}
                \HK_{n_0+e\mu}(R) = \eHK{R}p^{d(n_0+e\mu)}-c_1(\ell^\sharp)^{e}, e\geq 0.
            \end{gather*}
        
            The above implies that
            \begin{gather*}
                \begin{dcases}
                    \eHK{R} - c_1 = \HK_{n_0}(R),\\
                    \eHK{R}p^{dn_0} - c_1\ell^\sharp = \text{HK}_{n_0+\mu}(R).
                \end{dcases}
            \end{gather*}
            Replacing $c_1 = \eHK{R} - \HK_{n_0}(R)$ on the second equation, we obtain
            \begin{gather*}
                \eHK{R}p^{dn_0} - (\eHK{R} - \HK_{n_0}(R))\ell^\sharp = \text{HK}_{n_0+\mu}(R),
            \end{gather*}
            from which the formula follows.
        \end{proof}
        \begin{remark}
        The Han–Monsky algorithm is then the process of finding $n_0$ and $\mu$, and subsequently computing $\HK_{n_0}(R)$, $\HK_{\mu+n_0}(R)$, and $\ell^\sharp$. Note that, in general, $\mu$ can be chosen smaller than what the proof suggests.
        \end{remark}

\section{On the Hilbert–Kunz multiplicity of the \texorpdfstring{$A_1$}{A1} and \texorpdfstring{$A_2$}{A2} singularities, \texorpdfstring{$p>3$}{p gt 3}}\label{sec:strongwatanabeyoshida}
We provide a characteristic-free definition of the $A_1$ and $A_2$ singularities:
\begin{definition}[$A_1$ and $A_2$ singularities, \cite{greuel}]
    \label{df: Qd}
    Let $k$ be a field of positive characteristic $p>0$. The rings $R_{p,d} := \mathbb{F}_p[[x_0,...,x_d]]/(Q_d)$ where
    \begin{gather*}
    Q_d := \begin{cases}
    x_0x_1+\dots + x_{d-1}x_d, & \text{if $d$ is odd},\\
    x_0x_1 + \dots + x_{d-2}x_{d-1} + x_d^2, & \text{if $d$ is even},
    \end{cases}
    \end{gather*}
    are the $A_1$ singularities. The rings $S_{p,d} \coloneqq k[[x_0,...,x_d]]/(P_d)$ where
    \begin{equation*}
        P_d = \begin{cases}
        x_0x_1 + \dots + x_{d-1}^2 + x_d^3 & \text{if $d$ is odd},\\
        x_0x_1 + \dots + x_{d-2}x_{d-1} + x_d^3 & \text{if $d$ is even}.
        \end{cases}
    \end{equation*}
    are the $A_2$ singularities.
    \end{definition}
    \begin{remark}
    When characteristic is not $2$, a simple linear change of variables shows that $Q_d \sim x_0^2+\dots+x_d^2$ and $P_d \sim x_0^2+\dots+x_{d-1}^2+x_d^3$.
    \end{remark}
The point of this section is to show that 
\begin{gather}
    \label{fr: the_inequality}
    \eHK{S_{p,d}}>\eHK{R_{p,d}}, \forall p>3, d\geq 2
\end{gather} In §\ref{subsec: the_eHK_of_A1_and_A2}, we show that we can apply Theorem \ref{th: han_monsky_algorithm} to obtain a simpler formula for $\eHK{R_{p,d}}$ and $\eHK{S_{p,d}}$. In §\ref{subsec: technical_computations}, we carefully study two $k$-objects in the Han–Monsky ring to derive two inequalities that we bring together in §\ref{subsec: main_inequality} to derive the inequality (\ref{fr: the_inequality}).

\subsection{The Han–Monsky algorithm for the \texorpdfstring{$A_1$}{A1} and \texorpdfstring{$A_2$}{A2} singularities, \texorpdfstring{$p>3$}{p gt 3}}\label{subsec: the_eHK_of_A1_and_A2}

It turns out that in the case that the diagonal hypersurface is an $A_1$ or an $A_2$ singularity, we can choose $\mu=1$ and $n_0=0$ in Theorem \ref{th: han_monsky_algorithm}. Proving this fact is the content of the following lemma:
\begin{lemma}
\label{lm: conditions}
    Let $\beta = (\frac{1}{2},\dots,\frac{1}{2},\frac{1}{3})$, $\mu = 1$, and $p>3$. Then, the data $r,z,R,z',t,\ell^\sharp$ associated to $\beta,\mu$ and $p$ satisfy the hypotheses of Lemma \ref{lm: th53}. Also, $\sum r_i = 0$, so $\ell^\sharp = \ell_k(t)$. The same follows taking $\beta = (\frac{1}{2},\dots,\frac{1}{2})$, $\mu = 1$ and $p>2$.
\end{lemma}
\begin{proof}
Let us compute $r,z$ and $R$ first, according to Definition \ref{df: vectors}:
\begin{gather*}
    r = (0,\dots,0), z=\beta, R = (\left\lfloor\frac{p}{2}\right\rfloor,\dots,\left\lfloor\frac{p}{2}\right\rfloor,\left\lfloor\frac{p}{3}\right\rfloor).
\end{gather*}
Firstly, observe that $z = z^*$, given that $z_0 = \frac{1}{2}$. Thus, the hypotheses of the Theorem will be satisfied regardless of the parity of $\sum R_i$ once we prove that $z'$ is equivalent to $z$. 

When computing $z'$, two cases arise: either $p\equiv 1\mod{3}$ or $p\equiv 2\mod{3}$.
\begin{enumerate}
    \item If $p\equiv 1\mod{3}$, then $z' = (\frac{1}{2},\dots,\frac{1}{2},\frac{1}{3}) = z$.
    \item If $p\equiv 2\mod{3}$, then $z' = (\frac{1}{2},\dots,\frac{1}{2},\frac{2}{3})$. In this case, given that $s\geq 1$, and the first $s$ entries of $\beta$ are $1/2$, one sees that $z'$ is equivalent to $z$ by observing that $z' = (1-z_0,z_1,\dots,z_{s-1},1-z_s)$.
\end{enumerate}
The result follows.

As for $\beta = (1/2,\dots,1/2)$, $\mu = 1$ and $p>2$, it also happens that $z = z^*$, and so we do not have to bother about the parity of $\sum R_i$ either. Moreover, $p\equiv 1\pmod{2}$ for every $p>2$, so $z = z'$.
\end{proof}

\begin{corollary}
\label{cor: yoshida_theorem_mu_1}
For all $p>2$ and $d\geq 0$,
    \begin{gather*}
        e_{\text{HK}}(R_{p,d}) = 1 + \frac{\text{HK}_1(R_{p,d}) - p^d}{p^d-\ell^\sharp(R_{p,d})},
    \end{gather*}
    where $\ell^\sharp(R_{p,d}) = \alpha(\lambda_a^{d+1})$ for $a = \lfloor\frac{p}{2}\rfloor$. Similarly, for all $p>3$ and $d\geq 0$, it follows that
    \begin{gather*}
        e_{\text{HK}}(S_{p,d}) = 1 + \frac{\text{HK}_1(S_{p,d}) - p^d}{p^d-\ell^\sharp(S_{p,d})},
    \end{gather*}
    where $\ell^\sharp(S_{p,d}) = \alpha(\lambda_a^d\lambda_b)$ for $a = \lfloor\frac{p}{2}\rfloor$ and $b = \lfloor\frac{p}{3}\rfloor$.
\end{corollary}
\begin{proof}
    This is direct consequence of Theorem \ref{th: han_monsky_algorithm} and Lemma \ref{lm: conditions}, noting that $\HK_0(R) = 1$. 
\end{proof}

\begin{remark} 
    Other diagonal hypersurfaces for which the values $n_0=0$ and $\mu=1$ also work are, for example, all $x_0^{e_0}+\dots+x_d^{e_d}$ such that $e_0 = 2$ and $p\equiv \pm 1\mod{e_i}$ for every $i=1,\dots,d$. In this case, it happens that $z = z^*$, and therefore $z'$ is equivalent to both $z$ and $z^*$, and both sets of conditions in Lemma \ref{lm: th53} are satisfied. An example of such diagonal hypersurfaces is, for instance, the $A_3$ singularities with $p\geq 5$. More generally, the Hilbert–Kunz multiplicity of an $A_n$ singularity (see \cite{greuel}) verifies the formula in the theorem above whenever $p\equiv \pm 1\pmod{n+1}$.

    However, the formula does not work in full generality: one example of this is the $A_4$ surface singularity in characteristic 7, which has the equation $x^2+y^2+z^5$. The formula would give the value $43/25$, but it is known that this singularity has Hilbert–Kunz multiplicity $2-1/5 = 9/5$ (see, for instance, Example 4.1 \cite{wy01}).
\end{remark}
\subsection{The coefficients of the powers of two \texorpdfstring{$k$}{k}-objects}\label{subsec: technical_computations} In this section, we prove some necessary technical lemmas about two particular $k$-objects of the Han–Monsky ring that show up in the computation of $\HK_1(R_{p,d})$ and $\HK_1(S_{p,d})$. 

\begin{definition}
    In the Han–Monsky ring $\Gamma$ over a field $k$, we define the following two $k$-objects:
    \begin{gather*}
        \gamma := \left(\frac{k[x]}{(x^{p})},x^2\right) \quad \text{and} \quad  \eta := \left(\frac{k[x]}{(x^{p})},x^3\right).
    \end{gather*}
    Observe that $\gamma,\eta \in \Lambda_1$ (see \ref{lambdadef} for the definition of $\Lambda_e$). We will use the following notation for the coordinates of the powers of $\gamma$ over the $\lambda$-basis:
    \begin{gather*}
    \gamma^d = a_0^{(d)}\lambda_0+\dots+a_{p-1}^{(d)}\lambda_{p-1}.
    \end{gather*}
\end{definition}

\begin{proposition}
    \label{pr: gamma_formula}
    Let $k$ be a field of characteristic $p\geq 3$. Set $a := (p-1)/2$. Then,
    \begin{gather*}
    \gamma = \delta_{a}+\delta_{a+1} = 2\lambda_0-2\lambda_1+\dots+(-1)^{a-1}2\lambda_{a-1}+(-1)^{a}\lambda_{a}.
    \end{gather*}
\end{proposition}
\begin{proof}
Note that $\gamma = 2\delta_{p/2}$ (see \ref{df: representation_ring_rational_coefficients} for the definition of $\delta_{a/b}$).
\end{proof}

\begin{remark}
    It is worth mentioning that the results in this section were obtained in the first place as observations about the matrix representation of certain $k$-objects of the Han–Monsky ring. Yoshida introduced this matrix representation in \cite{yos}: roughly, given an element in $\Lambda_e\subset \Gamma$, it can be represented by the $p^e\times p^e$ symmetric matrix over the $\lambda$-basis that represents the action by multiplication. For instance, in this representation, if $p=2$, then we can express the $\lambda$-basis $\lambda_0,\lambda_1$ and $\lambda_2$ in $\Lambda_1$ as the following 3 matrices
    \begin{gather*}
    \lambda_0 \rightsquigarrow\left(\begin{array}{ccc}
    1 & 0 & 0 \\
    0 & 1 & 0 \\
    0 & 0 & 1
    \end{array}\right) \quad
    \lambda_1 \rightsquigarrow\left(\begin{array}{ccc}
    0 & 1 & 0 \\
    1 & 1 & 1 \\
    0 & 1 & 0
    \end{array}\right) \quad
    \lambda_2 \rightsquigarrow\left(\begin{array}{ccc}
    0 & 0 & 1 \\
    0 & 1 & 0 \\
    1 & 0 & 0
    \end{array}\right)
    \end{gather*}
    Therefore, according to the previous proposition, the $k$-object $\gamma \in \Lambda_1$ can be represented as the matrix
    \begin{gather*}
        \left(\begin{array}{rrr}
            2 & -1 & 0 \\
            -1 & 1 & -1 \\
            0 & -1 & 2
            \end{array}\right)
    \end{gather*}
    This matrix representation is utilized to address several questions about the behavior of $\eHK{R_{p,d}}$ in \cite{pssy25}. Although this is not the language used here, the matrix representation better illustrates many of the computations below. For instance, the first goal of the section is proving that the first $a+1$ coefficients of any power of $\gamma$ form a monotone sequence, where $a = (p-1)/2$ (see Proposition \ref{pr: decreasing_sequence}). These coefficients are the upper half of the first column of the powers of the matrix of $\gamma$.
\end{remark}

\begin{lemma}[\cite{yos}]
    \label{lm: product_gamma_lambda}
Let $k \leq  a = (p-1)/2$. Then, 
\begin{gather*} 
    \gamma\lambda_k = \begin{dcases}
    (-1)^k\left(2\sum_{i=0}^{a-k-1} (-1)^i\lambda_i + \sum_{i=a-k}^{a+k} (-1)^i\lambda_i\right) & (k < a) ;\\
    (-1)^a\delta_p & (k=a),
\end{dcases}
\end{gather*}
\end{lemma}
\begin{proof}
Explicit computation using Lemma \ref{lm: structure_Lambda_1} and Proposition \ref{pr: gamma_formula}. This formula is provided in matrix form in Section 4 of Yoshida's notes \cite{yos}.
\end{proof}

\begin{remark}
    Observe that for $p-1\geq k\geq a+1$, $\lambda_k\gamma = \lambda_{p-1}\lambda_{p-1-k}\gamma$, by \ref{lm: lambda_p_adic_split}\ref{lm: p_adic_opposite}.
\end{remark}
\begin{lemma}
    \label{lm: alpha_of_gamma_power_eta_and_gamma_power}
    Let $p>3$ and $d\geq 1$. Let $a,b,c\in\mathbb{Z}_{\geq 0}$ such that $a = (p-1)/2$ and $p = 3b+c$, and $c<3$. Then, 
    \begin{align*}
        \alpha(\gamma^d\eta) &= 3a_0^{(d)}+\dots +3a_{b-1}^{(d)}+ca_b^{(d)},\\
        \alpha(\gamma^{d+1}) &= 2a_0^{(d)}+\dots+2a_{a-1}^{(d)}+ a_a^{(d)}.
    \end{align*}
    \end{lemma}
    \begin{proof}
        Observe that, by definition and Lemma \ref{lm: product_gamma_lambda},
    \begin{gather*}
    \alpha(\gamma^{d+1}) = \sum_{i=0}^{p-1} (-1)^i a_i^{(d)} \alpha(\lambda_i\gamma) = \sum_{i=0}^{a-1} (-1)^i a_i^{(d)} (-1)^i 2 + (-1)^i a_a^{(d)} (-1)^i.
    \end{gather*}
    \end{proof}

\begin{lemma}
    \label{lm: gamma_power_reflexions_sums_delta_p_power}
    For every $d\geq 1$, $\gamma^d(\lambda_{p-1}+\lambda_0) = 2\cdot p^{d-1}\delta_p.$ Equivalently, writing $\gamma^d = \sum_{i=0}^{p-1} (-1)^ia_i^{(d)}\lambda_i$, then
\begin{gather*}
a_i^{(d)} + a_{p-i}^{(d)} = 2p^{d-1}.
\end{gather*}
In particular, this implies that $a_a^{(d)} = p^{d-1}$, for $a = \frac{p-1}{2}$.
\end{lemma}
\begin{proof}
One can check by Lemma \ref{lm: product_gamma_lambda} that $\gamma(\lambda_0+\lambda_{p-1}) = 2\delta_p$. Now, $(\lambda_0+\lambda_{p-1})^2 = 2(\lambda_0+\lambda_{p-1})$ by \ref{lm: lambda_p_adic_split}. By \ref{pr: delta_p_squared}, the result follows.
\end{proof}

\begin{lemma}[Theorem 4.1, \cite{yos}]
    \label{lm: gamma_power_recursive_formula}
    For $d\geq 2$,
    \begin{gather*}
        \gamma^d-p^{d-1}\delta_p = 2\sum_{r=0}^{a-1}((-1)^r\gamma^{d-1}\lambda_r - p^{d-2}\delta_p).
    \end{gather*}
    Equivalently, for $k = 1,\dots, a$
    \begin{gather*}
        a_{a-k}^{(d)}-p^{d-1} = 2\sum_{r = a-k}^{a-1} (a^{(d-1)}_{a-r}-p^{d-2}).
    \end{gather*}
\end{lemma}
\begin{proof}
Lemma \ref{lm: gamma_power_reflexions_sums_delta_p_power} applied twice yields
\begin{align*}
    \gamma^d - p^{d-1}\delta_p &= p^{d-1}\delta_p-\gamma^d\lambda_{p-1} = p^{d-1}\delta_p-\gamma(2p^{d-2}\delta_p-\gamma^{d-1})\\
     &= p^{d-1}\delta_p + 2\sum_{r=0}^{a-1}(-1)^r\lambda_r(\gamma^{d-1}-2p^{d-2}\delta_p) + (-1)^a\lambda_a(\gamma^{d-1}-2p^{d-2}\delta_p),
\intertext{and from Lemma \ref{lm: structure_Lambda_1} one can check that $(-1)^r\lambda_r\delta_p = \delta_p$, which yields}
\gamma^d - p^{d-1}\delta_p &= p^{d-1}\delta_p + 2\sum_{r=0}^{a-1}((-1)^r\lambda_r\gamma^{d-1}-2p^{d-2}\delta_p) + ((-1)^a\lambda_a\gamma^{d-1}-2p^{d-2}\delta_p)\\
    &= p^{d-1}\delta_p - (p-1)p^{d-2}\delta_p + 2\sum_{r=0}^{a-1}((-1)^r\lambda_r\gamma^{d-1}-p^{d-2}\delta_p) + ((-1)^a\lambda_a\gamma^{d-1}-2p^{d-2}\delta_p).
\end{align*}
To finish, note that $p^{d-1}\delta_p - (p-1)p^{d-2}\delta_p = p^{d-2}\delta_p,$ and adding the term in the rightmost side of the formula above,
\begin{gather*}
    p^{d-2}\delta_p + (-1)^a\lambda_a\gamma^{d-1}-2p^{d-2}\delta_p = (-1)^a\lambda_a\gamma^{d-1}-p^{d-2}\delta_p.
\end{gather*}
It follows from Lemma \ref{lm: gamma_power_reflexions_sums_delta_p_power} and the remark above that this is zero, so the result follows. 
\end{proof}
\begin{remark}
    Observe that, for $d\geq 2$, $2\sum_{r=0}^{a-1}((-1)^r\gamma^{d-1}\lambda_r - p^{d-2}\delta_p) = 2(\gamma^{d-1}-p^{d-2}\delta_p)\delta_a.$ Applying the result recursively, one obtains that $\gamma^d-p^{d-1}\delta_p = 2^{d-1}(\gamma-\delta_p)\delta_a^{d-1}$, for $d\geq 1$, which is an interesting identity on its own.
\end{remark}

From this lemma, we can now prove the result we were aiming for. We prove something slightly stronger, which is that the sequence is always strictly increasing for $d\geq 1$.
\begin{proposition}
\label{pr: decreasing_sequence}
    Fix $d\geq 1$. Then, the sequence $a_i^{(d)}$ is decreasing on $0\leq i \leq a$. It is strictly decreasing for $d\geq 2$.
\end{proposition}
\begin{proof}
    Let $k = 0,\dots,a-1$. Then, by Lemma \ref{lm: gamma_power_recursive_formula},
    \begin{multline*}
        a_{a-k}^{(d)} - a_{a-k-1}^{(d)} = 2\sum_{r=a-k}^a (a_{a-r}^{(d-1)}-p^{d-2}) - 2\sum_{r=a-k-1}^a(a_{a-r}^{(d-1)}-p^{d-2}) = 2(a_{k+1}^{(d-1)}-p^{d-2}).
    \end{multline*}
    Since $p^{d-2} = a_{a}^{(d-1)}$, the result follows from Proposition \ref{pr: gamma_formula} and induction on $d$.
\end{proof}
From this result, we obtain two useful and interesting inequalities.
\begin{proposition}
    \label{pr: first_inequality}
        For $d\geq 2$ and $p>3$ prime, $\HK_1(S_{p,d})-\HK_1(R_{p,d})\geq \HK_1(R_{p,d-1})-p^{d-1}$.
    \end{proposition}
    \begin{proof}
        Let $a,b,c$ like in \ref{lm: alpha_of_gamma_power_eta_and_gamma_power}. Since $d$ is fixed throughout the proof, there is no ambiguity in denoting $\gamma^d = \sum_{i = 0}^{p-1} a_i\lambda_i$. 
        By Lemma \ref{lm: alpha_of_gamma_power_eta_and_gamma_power}, rearranging the coefficients we obtain:
        \begin{gather*}
            \hs = \alpha(\gamma^d\eta) = \sum_{k=1}^{p}a_{\left\lfloor\frac{k+2}{3}\right\rfloor-1} \qandq \hr = \alpha(\gamma^{d+1}) = \sum_{k=1}^{p}a_{\left\lfloor \frac{k+1}{2}\right\rfloor-1}.
        \end{gather*}
    
        Thus,
            \begin{gather*}
                \hs-\hr = \sum_{k=1}^{p}\left(a_{\left\lfloor\frac{k+2}{3}\right\rfloor-1}-a_{\left\lfloor \frac{k+1}{2}\right\rfloor-1}\right).
            \end{gather*}
            Now, it happens that $\lfloor\frac{k+1}{2}\rfloor\geq\lfloor\frac{k+2}{3}\rfloor$, and the equality only holds for $k=1,2$ or $4$. This implies that $\lfloor\frac{k+1}{2}\rfloor>\lfloor\frac{k+2}{3}\rfloor$ for $k=3$ and $p\geq k\geq 5$. Hence, by Proposition \ref{pr: decreasing_sequence}  
            \begin{gather*}
                a_{\left\lfloor\frac{k+2}{3}\right\rfloor-1}-a_{\left\lfloor \frac{k+1}{2}\right\rfloor-1} \neq 0
            \end{gather*}
            for $k=3$ and $p\geq k\geq 5$. So,
            \begin{gather*}
                \sum_{k=1}^{p}\left(a_{\left\lfloor\frac{k+2}{3}\right\rfloor-1}-a_{\left\lfloor \frac{k+1}{2}\right\rfloor-1}\right) = (a_0-a_1) + \sum_{k=5}^p \left(a_{\left\lfloor\frac{k+2}{3}\right\rfloor-1}-a_{\left\lfloor \frac{k+1}{2}\right\rfloor-1}\right).
            \end{gather*}
            where every term of the sum on the right side is non-zero. Note that
            \begin{gather*}
                a_{\left\lfloor\frac{k+2}{3}\right\rfloor-1}-a_{\left\lfloor \frac{k+1}{2}\right\rfloor-1}\geq a_{\left\lfloor\frac{k+1}{2}\right\rfloor-2}-a_{\left\lfloor \frac{k+1}{2}\right\rfloor-1}
            \end{gather*}
            therefore
            \begin{align*}
                \sum_{k=5}^p \left(a_{\left\lfloor\frac{k+2}{3}\right\rfloor-1}-a_{\left\lfloor \frac{k+1}{2}\right\rfloor-1}\right) &\geq  \sum_{k=5}^p \left(a_{\left\lfloor\frac{k+1}{2}\right\rfloor-2}-a_{\left\lfloor \frac{k+1}{2}\right\rfloor-1}\right) \\
                &= 2(a_1-a_2)+2(a_2-a_3)+\dots+2(a_{a-2}-a_{a-1}) + (a_{a-1}-a_a)\\
                &\geq (a_1-a_2)+(a_2-a_3)+\dots+(a_{a-1}-a_a) = a_1-a_a
            \end{align*}
            Thus, to sum up, we just proved that
                    \begin{align*}
                        \hs-\hr &\geq (a_0-a_1)+(a_{1}-a_a) = a_0-a_a = \HK_1(R_{p,d-1})-p^{d-1}.
                    \end{align*}
    \end{proof}
    
    \begin{proposition}
    \label{pr: second_inequality}
        For $d\geq 2$ and $p>3$ prime, $p\HK_1(R_{p,d-1})\geq\HK_1(S_{p,d})$.
    \end{proposition}
    \begin{proof}
        We use the same notation in \ref{pr: first_inequality}. Then, by \ref{pr: decreasing_sequence},
        \begin{equation*}
            \HK_1(S_{p,d}) = 3a_{0}+\dots+3a_{b-1}+ca_{b} \leq 3a_{0}+\dots+3a_{0}+ca_{0} = pa_0 = p\HK_1(R_{p,d-1}).
        \end{equation*}
    \end{proof}

\subsection{The main inequality, \texorpdfstring{$p>3$}{p ge 3}}\label{subsec: main_inequality} We finally conclude in this section that the Hilbert–Kunz multiplicity of the $A_1$ and $A_2$ singularities are different. The following lemma lets us simplify our computation considerably by not allowing us to disregard $\ell^\sharp(R_{p,d})$ and $\ell^\sharp(S_{p,d})$:
\begin{lemma}[Theorem 4.6, \cite{hanmonsky}]
    \label{lm: ell_bound}
    If $d\geq 2$, $q= p^e$ and $k_i\leq q-1$, then $\alpha(\prod_{i = 0}^{d} \lambda_{k_i})\leq q^{d-2}$. In particular, $\ell^\sharp(S_{p,d}) \leq p^{d-2}$.
\end{lemma}

\begin{theorem}
    \label{cr: A_2_A_1_ineq_large_char}
    For $p> 3$ prime and $d\geq 2$, $\eHK{S_{p,d}}>\eHK{R_{p,d}}$.
\end{theorem}
\begin{proof}
    By Corollary \ref{cor: yoshida_theorem_mu_1} and Lemma \ref{lm: ell_bound},
\begin{gather*}
    \eHK{S_{p,d}} = \frac{\HK_1(S_{p,d})-p^d}{p^d-\ell^\sharp(S_{p,d})}\geq \frac{\HK_1(S_{p,d})-p^d}{p^d},\\
    \eHK{R_{p,d}}=\frac{\HK_1(R_{p,d})-p^d}{p^d-\ell^\sharp(R_{p,d})}\leq \frac{\HK_1(R_{p,d})-p^d}{p^d-p^{d-2}},
\end{gather*}
From Propositions \ref{pr: first_inequality} and \ref{pr: second_inequality}, it follows that
    \begin{multline}
        \label{fr: final_inequality_large_char}
        p^2(\HK_1(S_{p,d})-\HK_1(R_{p,d}))\geq p^2(\HK_1(R_{p,d-1})-p^{d-1}) >\\
        > p(\HK_1(R_{p,d-1})-p^{d-1})\geq \HK_1(S_{p,d})-p^d,
    \end{multline}
which implies that
\begin{gather*}
    \frac{\HK_1(S_{p,d})-p^d}{p^d}>\frac{\HK_1(R_{p,d})-p^d}{p^d-p^{d-2}}.
\end{gather*}
\end{proof}

\section{On the Hilbert–Kunz function of the \texorpdfstring{$A_1$}{A1} and \texorpdfstring{$A_2$}{A2} singularities, \texorpdfstring{$p=2,3$}{p23}}\label{sec:HKmultiplicity_A1_and_A2_char_2_3}
In this section, we will first explicitly compute the Hilbert–Kunz multiplicity of the $A_1$ and $A_2$ singularities in characteristics 2 and 3, in sections §\ref{subsec: A1_char_2}, §\ref{subsec: A2_char_2} and §\ref{subsec:A2_in_char_3}, respectively. In the last part, in §\ref{subsec: main_inequality_char_2_and_3}, we show that $\eHK{S_{2,d}}>\eHK{R_{2,d}}$ for $p=2,3$, $d\geq 2$.

\subsection{The \texorpdfstring{$A_1$}{A1} singularity in characteristic 2}\label{subsec: A1_char_2}
The theory on the arithmetic of the Han–Monsky ring in characteristic 2 developed in Section \ref{sec:han_monsky_ring} provides a powerful way of computing the Hilbert–Kunz function of the $A_1$ singularity thanks to the fact that it's an orthogonal basis of the $\Lambda_e\otimes_\mathbb{Z} \mathbb{Q}$. When $d$ is even, Yoshida \cite{yos} computed $\eHK{R_{2,d}}$ and conjectured the formula in the case $d$ is odd. The goal of this section is to compute $\eHK{R_{2,d}}$ for every $d$. First, a lemma about an object of the Han–Monsky ring:
\begin{lemma}
    \label{lm: M_e_over_delta_basis}
Let $k$ be a field of characteristic 2. For $e\geq 1$, set $\mu_e := \left(k[x,y]/(x^{2^e},y^{2^e}),xy\right)$. Then $\mu_e = \delta_{2^e}+2\delta_{1}+\dots+2\delta_{2^e-1}.$
\end{lemma}
\begin{proof}
    Given that the $T$-action is given by product by a monomial, it is easy to see that the orbits of the monomials $1,x,...,x^{2^n-1},y,...,y^{2^n-1}$ under the action of $T$ produce disjoint $k$-vector subspaces that span the $k$-object – more specifically, the orbits are the following:
    \begin{equation*} 
    \begin{array}{rllllll}
        1 &\mapsto xy &\mapsto \dots &\mapsto x^{2^e-1}y^{2^e-1} &\mapsto 0,\\[.5em]
        x^i &\mapsto x^{i+1}y &\mapsto \dots &\mapsto x^{2^e-1}y^{2^e-i-1}&\mapsto 0,\\[.5em]
        y^j &\mapsto xy^{j+1} &\mapsto \dots &\mapsto x^{2^e-j-1}y^{2^e-1}&\mapsto 0.
    \end{array}
    \end{equation*}
    According to this decomposition, one sees that $x^i$ (or $y^i$) generates a cyclic $k[T]$-submodule whose isomorphism class is $\delta_{2^e-i}$.
    \end{proof}
\begin{proposition}
\label{pr: A_1_in_representation_ring}
    Let $k$ be a field of characteristic 2. Then,
    \begin{gather*}
        \left(\frac{k[[x_0,\dots,x_d]]}{(x_0^{[2^e]},\dots,x_d^{[2^e]})},Q_d\right) = \begin{cases} \mu_e^m\cdot 2\delta_{2^{e-1}} & (d=2m)\\
        \mu_e^m & (d=2m-1)
        \end{cases}
    \end{gather*}
    where $Q_d$ is as in Definition \ref{df: Qd} and $\mu_e$ as in the previous lemma.
\end{proposition}
\begin{proof}
    If $d = 2m$, $m\geq 1$, then $Q_d = x_0x_1+\dots+x_{d-2}x_{d-1}+x_d^2$. This implies that the $k$-object in question decomposes as:
    \begin{gather*}
        \frac{k[x,y]}{(x^{[2^e]},y^{[2^e]})}\otimes \stackrel{(m)}\dots \otimes \frac{k[x,y]}{(x^{[2^e]},y^{[2^e]})}\otimes \frac{k[x]}{(x^{2^e})},
    \end{gather*}
    where the $k$-object $\frac{k[x,y]}{(x^{[2^e]},y^{[2^e]})}$ has $T$-action $xy$, the $k$-object $\frac{k[x]}{(x^{[2^e]})}$ has the $T$-action $y^2$.
    
    If $d = 2m-1$, $m\geq 1$, then $Q_d=x_0x_1+\dots+x_{d-3}x_{d-2}+x_{d-1}x_d$ and the $k$-object decomposes as
    \begin{gather*}
    \frac{k[x,y]}{(x^{[2^e]},y^{[2^e]})}\otimes \stackrel{(m)}\dots \otimes \frac{k[x,y]}{(x^{[2^e]},y^{[2^e]})},
    \end{gather*}
    where the $k$-object $\frac{k[x,y]}{(x^{[2^e]},y^{[2^e]})}$ has $T$-action $xy$.
\end{proof}

The aim of the following lemmas is expressing $\mu_e$ over the $\sigma$-basis:

\begin{lemma}
\label{lm: M_e_over_lambda_basis}
For $e\geq 1$, $\mu_e = \sum_{k=0}^{2^e-1} (-1)^{k}(2(2^e-k)-1)\lambda_{k}$.
\end{lemma}
\begin{proof}
We use the change of variables in the remark after Definition \ref{def: lambda} on the expression in Lemma \ref{lm: M_e_over_delta_basis},
\begin{align*}
    2\delta_1+...+2\delta_{2^e-1}+\delta_{2^e}&= 2\lambda_0 +2(\lambda_0-\lambda_1) + \cdots +2(\lambda_0- \dots +\lambda_{2^e-2}) +(\lambda_0-\dots -\lambda_{2^e-1})\\
    &= \left[2\cdot (2^e-1)\lambda_0 -2\cdot (2^e-2)\lambda_1 + ... +2\lambda_{2^e-2}\right] +\left[\lambda_0-\dots - \lambda_{2^e-1}\right]\\
    &=\sum_{k=0}^{2^e-1} (-1)^{k}(2(2^e-k-1)+1)\lambda_{k}.
\end{align*}
\end{proof}

\begin{corollary}
\label{cor: A1_recursion_char_2}
    The expression $\mu_e$ verifies the following recursive formula:
    \begin{gather*}
        \mu_e = (\lambda_0-\lambda_{2^{e-1}})\mu_{e-1} + 2^{2e-1}\sigma_{1,e-1}
    \end{gather*}
    for every $e\geq 2$. 
\end{corollary}
\begin{proof}
We use the expression given by Lemma \ref{lm: M_e_over_lambda_basis},
    \begin{align*}
    \mu_e &= \sum_{k=0}^{2^e-1} (-1)^{k}(2(2^e-k)-1)\lambda_{k}\\
    &=\sum_{k=0}^{2^{e-1}-1} (-1)^{k}(2(2^e-k)-1)\lambda_{k}
        + \sum_{k=2^{e-1}}^{2^e-1} (-1)^{k}(2(2^e-k)-1)\lambda_{k}.
    \end{align*}
    The second summand is $-\lambda_{2^{e-1}}\mu_{e-1}$ by \ref{lm: lambda_p_adic_split}; as for the first summand,
    \begin{align*}
        \sum_{k=0}^{2^{e-1}-1} (-1)^{k}(2(2^e-k)-1)\lambda_{k} &= \sum_{k=0}^{2^{e-1}-1} \left[(-1)^{k}(2(2^{e-1}-k)-1) + (-1)^k2\cdot 2^{e-1}\right]\lambda_{k}=\\
        &=\sum_{k=0}^{2^{e-1}-1} (-1)^{k}(2(2^{e-1}-k)-1)\lambda_k + 2^e\sum_{k=0}^{2^{e-1}-1} (-1)^k\lambda_k =\\
        &= \mu_{e-1} + 2^e2^{e-1}\sigma_{1,e-1}.
    \end{align*}
    and the result follows from the same lemma and \ref{pr: sigmaproperties}\ref{pr: sigma_1_delta_and_lambda}.

    Alternatively, one can deduce the same recursive formula from Lemma \ref{lm: M_e_over_delta_basis} directly observing that $(\mu_e-\delta_{2^e}) = (\lambda_0+\lambda_{2^{e-1}})(\mu_{e-1}-\delta_{2^{e-1}})+2^{e-1}\delta_{2^{e-1}}$, using the definition of $\lambda_i$ and Lemma 3.3 in \cite{hanmonsky}.
\end{proof}

\begin{remark}
    It is possible to deduce Corollary \ref{cor: A1_recursion_char_2} with the notation and results in \cite{mt06}. Using the formula for $\varphi_{fg}$ provided in Proposition 4.2 of that paper, one directly obtains that $\varphi:= \varphi_{xy} = 2t-t^2$, which is immediately seen to satisfy the functional equations $4\varphi(t/2) = \varphi+ 2t$ and $4\varphi((t+1)/2) = \varphi + 3$. Using arguments very similar to those in Example 1, one concludes that $4S(a) = a + \lambda_1 a + 2(\lambda_0-\lambda_1)\Delta$, where $a = \mathscr{L}(\varphi)$. It is not difficult to see that $(\lambda_0-\lambda_1)\Delta = 2S(\Delta)$, so Corollary \ref{cor: A1_recursion_char_2} follows.
\end{remark}
\begin{proposition}
\label{pr: M_e_over_sigma_basis}
Given $e\geq 1$, then $\mu_e = 2^e(\sigma_{0,e} + 2^e\sigma_{1,e}+\sum_{k=1}^{e-1} 2^k\sigma_{2^k+1,e}).$
\end{proposition}
\begin{proof}
We prove it by induction on $e$. By Proposition \ref{cor: A1_recursion_char_2}, $M_1 = 3\lambda_0-\lambda_1 = 2\sigma_{0,1}+4\sigma_{1,1},$ so the result holds when $e=1$ (where we regard $\sum_{k=1}^{0}2^k\sigma_{2^k+1,1} = 0$.)

To show the general $e\geq 1$ case, we will employ the recursive formula in Proposition \ref{cor: A1_recursion_char_2}, the property that $(\lambda_0+\lambda_{2^e})\sigma_{k,e} = 2\sigma_{k,e+1}$, and Proposition \ref{pr: sigmaproperties}\ref{pr: recursive_split}:
\begin{align*}
\mu_{e+1} &= (\lambda_0 + \lambda_{2^{e}})\mu_e + 2^{2(e+1)-1}\sigma_{1,e}\\
    &=2^e(\lambda_0+\lambda_{2^e})(\sigma_{0,e} + 2^e\sigma_{1,e}+\sum_{k=1}^{e-1} 2^k\sigma_{2^k+1,e}) +2^{2e+1}\sigma_{1,e}\\
    &=2^e\cdot 2(\sigma_{0,e+1} + 2^e\sigma_{1,e+1} + \sum_{k=1}^{e-1}2^k\sigma_{2^k+1,e+1}) + 2^{2e+1}(\sigma_{1,e+1}+\sigma_{2^e+1,e+1})\\
    &=2^{e+1}(\sigma_{0,e+1} + 2^e\sigma_{1,e+1} + \sum_{k=1}^{e-1}2^k\sigma_{2^k+1,e+1}+2^e\sigma_{1,e+1}+2^e\sigma_{2^e+1,e+1}).
\end{align*}
The result follows from the last line.
\end{proof}
\begin{theorem}
\label{th: HK_function_A_1_char_2}
For $d\geq 2$,
\begin{gather*}
    \text{HK}_{e}(R_{2,d}) =\begin{dcases}
        \frac{2^m}{2^m-1}2^{de}-\frac{(2^{m-1})^e}{2^m-1}, &\text{ if } d=2m-1 ;\\
        \frac{2^m+1}{2^m}2^{de}, &\text{ if } d=2m.
    \end{dcases}
\end{gather*}
where $R_{p,d}$ is as in Definition \ref{df: Qd}. In particular, the Hilbert–Kunz multiplicity of an $A_1$ singularity in characteristic 2 is
\begin{gather*}
    \eHK{R_{2,d}} = \begin{dcases}
        \frac{2^m}{2^m-1}, &\text{ if } d=2m-1 ;\\
        \frac{2^m+1}{2^m}, &\text{ if } d=2m.\\
    \end{dcases}
\end{gather*}
\end{theorem}
\begin{proof}
    Firstly, let $d=2m-1$. The result follows from Proposition \ref{pr: M_e_over_sigma_basis}, using idempotency and orthogonality of the $\sigma$-elements and Proposition \ref{pr: sigmaproperties}\ref{pr: alpha_of_sigma}:
    \begin{align*} 
    \text{HK}_e(R_{2,d}) &= \alpha\left(\mu_e^m\right) = 2^{em}\alpha\left(\sigma_{0,e} + 2^{em}\sigma_{1,e} + \sum_{k=1}^{e-1}2^{km}\sigma_{2^k+1,e}\right)\\
    &= 2^{em}\sum_{k=0}^e 2^{km}\frac{1}{2^e} = 2^{e(m-1)}\sum_{k=0}^e 2^{km} = \frac{2^{(e+1)m}-1}{2^m-1}.
    \end{align*}
    So the formula follows in the odd-dimensional case.

Now, for the even-dimensional case, $d=2m$, observe that $R_{2,d}/\m^{[2^e]} = 2\delta_{2^{e-1}}\mu_e^m$, and by Proposition \ref{pr: sigmaproperties}\ref{pr: sigma_1_delta_and_lambda}, we have that $\delta_{2^{e-1}} = 2^{e-1}\sigma_{1,e-1} = 2^{e-1}(\sigma_{1,e}+\sigma_{2^{e-1}+1,e}).$ Then, we proceed similarly – by Proposition \ref{pr: A_1_in_representation_ring}, we have that $\HK_e(R_{2,2m}) = \alpha(\mu_e^m\cdot 2\delta_{2^{e-1}}),$ and
\begin{align*}
    \mu_e^m\cdot 2\delta_{2^{e-1}} &= 2^{em}(\sigma_{0,e} + 2^{em}\sigma_{1,e} + \sum_{k=1}^{e-1}2^{km}\sigma_{2^k+1,e})\cdot 2^e(\sigma_{1,e}+\sigma_{2^{e-1}+1,e})\\
    &= 2^{e(m+1)}(2^{em}\sigma_{1,e} + 2^{(e-1)m}\sigma_{2^{e-1}+1,e}).
\end{align*}
Again, using that $\alpha(\sigma_{k,e}) = \frac{1}{2^e}$, from applying $\alpha$ to the above, we obtain that $\text{HK}_n(R_{2,d})$ is
\begin{equation*}
    2^{e(m+1)}(2^{em}\frac{1}{2^e}+2^{(e-1)m}\frac{1}{2^e}) = 2^{em}(2^{em}+2^{(e-1)m}) = 2^{nd} + \frac{1}{2^m}2^{dn} = \frac{2^m+1}{2^m}2^{dn}.
\end{equation*}
\end{proof}

\subsection{The \texorpdfstring{$A_2$}{A2} singularity in characteristic 2}\label{subsec: A2_char_2}
We now compute $\eHK{S_{2,d}}$. The main difficulty of this computation is dealing with the fact that the equation of the $A_2$ singularity has an $x^3$ term in its equation, and the $k$-object $k[x]/(x^{2^e})$ with $T$-action $x^3$ has a complicated decomposition in the Han–Monsky ring. However, the orthogonality property of the $\sigma$-basis allows us to disregard most of the coefficients of the decomposition. The first goal is obtaining a recursive formula for this $k$-object.

\begin{proposition}
    \label{pr: A_2_in_representation_ring}
        Let $k$ be a field of characteristic 2. Then, in the Han–Monsky ring
        \begin{gather*}
            \left(\frac{k[[x_0,\dots,x_d]]}{(x_0^{[2^e]},\dots,x_d^{[2^e]})},P_d\right) = \begin{cases} \mu_e^m3\delta_{2^{e}/3} & (d=2m)\\
            \mu_e^{m-1}2\delta_{2^{e-1}}3\delta_{2^e/3} & (d=2m-1)
            \end{cases}
        \end{gather*}
        where $P_d$ was defined in Definition \ref{df: Qd}, and $\mu_e$ is as defined in Lemma \ref{lm: M_e_over_delta_basis}.
    \end{proposition}
    \begin{proof}
        Let $d=2m+1$. Then, $P_d=x_0x_1+\dots+x_{d-3}x_{d-2}+x_{d-1}^2+x_d^3$ and the object decomposes as
        \begin{gather*}
        \frac{k[x,y]}{(x^{[2^e]},y^{[2^e]})}\otimes \dots \otimes \frac{k[x,y]}{(x^{[2^e]},y^{[2^e]})} \otimes \frac{k[y]}{(y^{[2^e]})} \otimes \frac{k[x]}{(x^{[2^e]})},
        \end{gather*}
        where the $k$-object $k[x,y]/(x^{[2^e]},y^{[2^e]})$ has $T$-action $xy$, the $k$-object $k[y]/(y^{[2^e]})$ has the $T$-action $y^2$, and finally the $k$-object $k[x]/(x^{[2^e]})$ has $T$-action $x^3$. We already know that the first is $\mu_e$ and the second $2\delta_{2^{e-1}}$. The third $k$-object is $3\delta_{2^e/3}$, as one deduces from Lemma \ref{lm: lemmarepring} for $s = 0$.

        If $d=2m$, then $P_d = x_0x_1+\dots+x_{d-2}x_{d-1}+x_d^3$ and the result follows from the same discussion.
    \end{proof}

    \begin{lemma}
        \label{lm: A2_char2_recursion}
        For $e\geq 1$, let $\xi_e := 3\delta_{2^e/3}$. Then, $\xi_1 = 2\lambda_0 = 2\sigma_{0,1}+2\sigma_{1,1},\xi_2 = 3\lambda_0-\lambda_1 = 2\sigma_{0,2}+4\sigma_{1,2}+2\sigma_{2,2}+4\sigma_{3,2}$,
    \begin{gather*}
    \xi_{e+1}-\lambda_{2^e-1}\xi_e = 3\delta_{2^e}, e\geq 1.
    \end{gather*}
    \end{lemma}
    \begin{proof}
    Let $e\geq 1$, and denote $2^e = 3d+r$ and $2^{e+1} = 3D+R$, where $d,r, D,R\in \mathbb{Z}$ and $0\leq r,R<3$. Then, from Definition \ref{df: representation_ring_rational_coefficients},
    \begin{align*}
        \xi_{e+1} &= 3\lambda_0-3\lambda_1+\dots+(-1)^{D-1}3\lambda_{D-1}+(-1)^DR\lambda_{D},\\
        \xi_e &= 3\lambda_0-3\lambda_1+\dots+(-1)^{D-1}3\lambda_{d-1}+(-1)^dr\lambda_{d}.
    \end{align*}
    Then, using \ref{lm: lambda_p_adic_split}\ref{lm: p_adic_opposite} one obtains
    \begin{multline}
        \label{fm: proving_recurrence_of_Ce}
    \xi_{e+1}-\lambda_{2^e-1}\xi_e = (3\lambda_0-\dots +(-1)^DR\lambda_D)-\\
    -((-1)^dr\lambda_{2^e-d-1}+(-1)^{d-1}3\lambda_{2^e-d-2}+\dots-3\lambda_{2^e-2}+3\lambda_{2^e-1}).
    \end{multline}
    One can check that $D = 2^e-(d+1)$ and $R+r = 3$, which implies that $(-1)^DR\lambda_D-(-1)^dr\lambda_{2^e-d-1} = (-1)^D3\lambda_D$, which implies that the RHS in (\ref{fm: proving_recurrence_of_Ce}) is $3\delta_{2^e}$.
    \end{proof}

        Now that we have a recursive formula for $\xi_e$, we can compute some coefficients over the $\sigma$-basis. Since most of the coefficients get annihilated by the orthogonality property in the computation of the Hilbert–Kunz function of the $A_2$ singularity (see Theorem \ref{th: HK_function_A_2_char_2}), we don't need to compute the complete decomposition.
    \begin{lemma}
        \label{lm: C_e_critical_coefficients}
        Let $\xi_e$ as in Lemma \ref{lm: A2_char2_recursion}. Let $a_{k,e}\in \mathbb{Q}$ such that $\xi_e = \sum_{k=0}^{2^e-1} a_{k,e}\sigma_{k,e}$. Then $a_{0,e} = 2$, $a_{1,e} = 2^{e+1}$, $a_{2^k+1,e} = 2^{k+1}$ for $k=1,\dots,e-2$, $e\geq 2$.
    \end{lemma}
\begin{proof}    
        Throughout the proof, we use that $\sigma_{k,e}\xi_e = a_{k,e}\sigma_{k,e}$, and the rest of the orthogonality and idempotency properties (\ref{pr: sigmaproperties}\ref{pr: idem_orth}, \ref{pr: extracting_coordinates_orthogonality} and \ref{pr: inter_orthogonality}) without further mention. Also, $e\geq 2$.
    \begin{enumerate}
        \item We start computing $a_{0,e}$ and $a_{1,e}$, for $e\geq 2$. Since $\delta_{2^e} = 2^e\sigma_{1,e}$ by \ref{pr: sigmaproperties}\ref{pr: sigma_1_delta_and_lambda}, it follows that $\sigma_{0,e+1}\delta_{2^e} = 0$, and $\sigma_{1,e+1}\delta_{2^e} = 2^e\sigma_{1,e+1}$. Also, by Lemma \ref{lm: lambda_sigma_product_e+1}, $\sigma_{0,e+1}\lambda_{2^e-1} = \sigma_{0,e+1}$ and $\sigma_{1,e+1}\lambda_{2^e-1} = -\sigma_{1,e+1}$. Thus, multiplying the recurrence in Lemma \ref{lm: A2_char2_recursion} by $\sigma_{0,e+1}$ and $\sigma_{1,e+1}$, we obtain in each case
    \begin{align*}
    \sigma_{0,e+1}\xi_{e+1}-\sigma_{0,e+1}\lambda_{2^e-1}\xi_e = 0 &\Longrightarrow a_{0,e+1}-a_{0,e} = 0,\\
    \sigma_{1,e+1}\xi_{e+1}-\sigma_{1,e+1}\lambda_{2^e-1}\xi_e = 3\cdot 2^e\sigma_{1,e+1} &\Longrightarrow a_{1,e+1}+a_{1,e} = 3\cdot 2^e.
    \end{align*}
    Given that $a_{0,1} = a_{1,1} = 2$ by \ref{lm: A2_char2_recursion}, we have that $a_{0,e} = 2$ and $a_{1,e} = 2^e$ for $e\geq 1$.
    \item In order to compute $a_{2^k+1,e}$ for $1\leq k\leq e-2$, we first need to compute $a_{2^{e-1}+1,e}$ for every $e\geq 2$. Using the recurrence in \ref{lm: A2_char2_recursion} again, 
    \begin{gather}
        \label{fr: base_case_product}
        a_{2^{e-1}+1,e}\sigma_{2^{e-1}+1,e} = \sigma_{2^{e-1}+1,e}\xi_{e} = \lambda_{2^{e-1}-1}\sigma_{2^{e-1}+1,e}\xi_{e-1} + 3\sigma_{2^{e-1}+1,e}\delta_{2^{e-1}}.
        \end{gather}
        
        As for the product $\sigma_{2^{e-1}+1,e}\delta_{2^{e-1}}$ above, by \ref{pr: sigmaproperties}\ref{pr: sigma_1_delta_and_lambda}, we have that 
        \begin{gather*}
            \sigma_{2^{e-1}+1,e}\delta_{2^{e-1}} = \sigma_{2^{e-1}+1,e}2^{e-1}\sigma_{1,e-1} = 2^{e-1}\sigma_{2^{e-1}+1,e},
        \end{gather*}
        using that $\sigma_{2^{e-1}+1,e}\sigma_{1,e-1} = \sigma_{2^{e-1}+1,e}$.

        About the first summand in (\ref{fr: base_case_product}), observe that for $e\geq 2$, $\lambda_{2^{e-1}-1}\sigma_{2^{e-1}+1,e} = \sigma_{2^{e-1}+1,e}$ by Lemma \ref{lm: lambda_sigma_product_e+1}, so we continue with the RHS of (\ref{fr: base_case_product}):
        \begin{gather*}
        3\sigma_{2^{e-1}+1,e}\delta_{2^{e-1}}-\sigma_{2^{e-1}+1,e}\lambda_{2^{e-1}-1}\xi_{e-1} = 3\cdot 2^{e-1}\sigma_{2^{e-1}+1,e}-\sigma_{2^{e-1}+1,e}\xi_{e-1}.
        \end{gather*}
        Finally, from \ref{pr: sigmaproperties}\ref{pr: recursive_split},
        \begin{align*}
        \sigma_{2^{e-1}+1,e}\xi_{e-1} &= \sigma_{2^{e-1}+1,e}\sum a_{s,e-1}\sigma_{s,e-1} \\
        &= \sigma_{2^{e-1}+1,e}\left(\sum_{s=0}^{2^{e-1}-1} a_{s,e-1}\sigma_{s,e} + \sum_{s=2^{e-1}}^{2^{e}-1} a_{s,e-1}\sigma_{2^{e-1}+s,e}\right) = a_{1,e-1}\sigma_{2^{e-1}+1,e}.
        \end{align*}
        We conclude from this and (\ref{fr: base_case_product}) that $a_{2^{e-1}+1,e} = 3\cdot 2^{e-1}-a_{1,e-1} = 2\cdot 2^{e-1} = 2^{e}$, and therefore,
        \begin{gather}
            \label{fm: a2e_11e2e}
        a_{2^{e-1}+1,e} = 2^{e}, e\geq 2.
        \end{gather}
    \item We now compute $a_{2^k+1,e}$ for $1\leq k\leq e-2$ and $e\geq 2$. Note that, for those $k$ and $e$, it happens again that $\sigma_{2^k+1,e+1}\delta_{2^e} = 0$, so multiplying the recurrence in \ref{lm: A2_char2_recursion} by $\sigma_{2^k+1,e+1}$, we obtain
    \begin{gather*}
        \sigma_{2^k+1,e+1}\xi_{e+1}-\sigma_{2^k+1,e+1}\lambda_{2^e-1}\xi_e = 0 \Longrightarrow a_{2^k+1,e+1}-a_{2^k+1,e} = 0.
    \end{gather*}
    This implies that for all $1\leq k\leq e-2$, $a_{2^k+1,e} = a_{2^k+1,k+1}$, and we know by the previous part (setting $e = k+1$ in (\ref{fm: a2e_11e2e})) that $a_{2^k+1,k+1} = 2^{k+1}$ for $k\geq 2$. Therefore, we conclude that
    \begin{gather*}
    a_{2^k+1,e} = 2^{k+1}, 1\leq k\leq e-2, e\geq 2.
    \end{gather*}
    \end{enumerate}
    \end{proof}
Table \ref{tab: critical_coefficients_x_cubed} summarizes the results in the previous lemmas. 
\begin{table}[h!]
    \label{tab: critical_coefficients_x_cubed}
\centering
\begin{tabular}{c|c|c}
$a_{0,e}$ & $a_{1,e}$ & $a_{2^k+1,e},1\leq k\leq e-1$ \\ \hline
$2$   & $2^e$   & $2^{k+1}$  \\
\end{tabular}
\vspace{1em}
\caption{Critical coefficients of $x^3$.}
\label{tab: critical_coefficients_of_x3}
\end{table}
\begin{theorem}
    \label{th: HK_function_A_2_char_2}
    The Hilbert--Kunz function of an $A_2$ singularity in characteristic 2 is
    \begin{gather*}
        \HK_e(S_{2,d}) = \begin{dcases}
        \frac{2^{m+1}+1}{2^{m+1}-1}2^{de}-\frac{2}{2^{m+1}-1}2^{e(m-1)} &(d=2m) ;\\
        \frac{2^m+1}{2^m}2^{de} &(d=2m-1);
    \end{dcases}
\end{gather*}
with $S_{p,d}$ as in Definition \ref{df: Qd}. In particular, the Hilbert–Kunz multiplicity of an $A_2$ singularity in characteristic 2 is
\begin{gather*}
    \eHK{S_{2,d}} = \begin{dcases}
    \frac{2^{m+1}+1}{2^{m+1}-1} & (d=2m) ;\\
    \frac{2^{m-1}+1}{2^{m-1}} & (d=2m-1) ;\\
\end{dcases}
\end{gather*}
\end{theorem}
\begin{proof}
    We treat even- and odd-dimensional cases separately.
    
    Firstly, let $d=2m+1$. By \ref{pr: A_2_in_representation_ring}, $\HK_e(S_{3,2m+1}) = \alpha\left(\mu_e^m\xi_e\right)$. Let $\xi_e$ as in \ref{lm: C_e_critical_coefficients}, and $\mu_e$ as in \ref{pr: M_e_over_sigma_basis}, so
    \begin{align*}
        \mu_e^m\xi_e&= 2^{em}(\sigma_{0,e}+2^{em}\sigma_{1,e}+\sum_{k=1}^{e-1} 2^{km}\sigma_{2^k+1,e})\sum_{k=0}^{2^e-1} a_{k,e}\sigma_{k,e}\\
        &= 2^{em}(a_{0,e}\sigma_{0,e}+2^{em}a_{1,e}\sigma_{1,e}+\sum_{k=1}^{e-1}2^{km}a_{2^k+1,e}\sigma_{2^k+1,e}).
    \end{align*}
    Using Lemma \ref{lm: C_e_critical_coefficients}, the orthogonality properties of $\sigma$-basis, \ref{pr: sigmaproperties}, and the linearity of $\alpha$,
    \begin{align*}
        \alpha(\mu_e^m\xi_e) &= 2^{em}\alpha(a_{0,e}\sigma_{0,e}+2^{em}a_{1,e}\sigma_{1,e}+\sum_{k=1}^{e-1}2^{km}a_{2^k+1,e}\sigma_{2^k+1,e})\\
        &= 2^{em}(2\frac{1}{2^e}+2^{em}2^e\frac{1}{2^e}+\sum_{k=1}^{e-1}2^{km}2^{k+1}\frac{1}{2^e})\\
        &= 2^{em}(2+2^{e(m+1)}+2\sum_{k=1}^{e-1}2^{k(m+1)}) \\
        &= 2^{em}(-2^{e(m+1)}+2\sum_{k=0}^{e}2^{k(m+1)}) = 2^{em}(-2^{e(m+1)}+2\frac{2^{(e+1)(m+1)}-1}{2^{m+1}-1}).
    \end{align*}
    Thus,
    \begin{gather*}
        \HK_e(S_{2,d}) = 2^{em}(-2^{e(m+1)}+2\frac{2^{(e+1)(m+1)}-1}{2^{m+1}-1}) = -2^{de}+2\frac{2^{m+1}2^{de}-2^{e(m-1)}}{2^{m+1}-1},
    \end{gather*}
    from which the result follows for $d=2m+1$.
    
    Let now $d=2m$. Recall that by Proposition \ref{pr: sigmaproperties}, we have that $\delta_{2^{e-1}} = 2^{e-1}\sigma_{1,e-1} = 2^{e-1}(\sigma_{1,e}+\sigma_{2^{e-1}+1,e})$, so $2\delta_{2^{e-1}} = 2^e(\sigma_{1,e}+\sigma_{2^{e-1}+1,e})$. Now, by Proposition \ref{pr: A_2_in_representation_ring}, we have that $\HK_e(S_{3,2m}) = \alpha\left(2\mu_e^{m-1}\delta_{2^{e-1}}\xi_e\right)$, so
    \begin{align*}
        \mu_e^{m-1}2\delta_{2^{e-1}}\xi_e &= 2^{e(m-1)}(\sigma_{0,e}+2^{e(m-1)}\sigma_{1,e}+\sum_{k=1}^{e-1} 2^{k(m-1)}\sigma_{2^k+1,e})2^e(\sigma_{1,e}+\sigma_{2^{e-1}+1,e})\sum_{k=0}^{2^e-1} a_{k,e}\sigma_{k,e}\\
        &=2^{em}2^{e(m-1)}2^ea_{1,e}\sigma_{1,e} + 2^{em}2^{(e-1)(m-1)}2^ea_{2^{e-1}+1,e}\sigma_{2^{e-1}+1,e}.
    \end{align*}
    Thus, applying $\alpha$ to the above, and using Lemma \ref{lm: C_e_critical_coefficients},
    \begin{gather*}
        \HK_e(S_{2,d}) = 2^{de}2^e\frac{1}{2^e}+2^{de}\frac{1}{2^m}2^e\frac{1}{2^e} = \frac{2^m+1}{2^m}2^{de}.
    \end{gather*}
\end{proof}

\subsection{The \texorpdfstring{$A_1$ and $A_2$}{A1 and A2} singularity in characteristic 3}\label{subsec:A2_in_char_3}
Now, we compute the Hilbert–Kunz multiplicity of the $A_1$ and $A_2$ singularity in characteristic 3. Yoshida already computed the first one in \cite{yos}, so we only need to compute the second one.
\begin{theorem}
    \label{th: eHK_A_1_and_A_2_in_char_3}
    For $d\geq 2$,
    \begin{gather*}
        \eHK{R_{3,d}} = 1+\frac{3\cdot 2^d}{3^{d+1}-2^d+(-1)^d} \qandq \eHK{S_{3,d}} = 1+\left(\frac{2}{3}\right)^{d-1}
    \end{gather*}
\end{theorem}
\begin{proof}
    The Hilbert–Kunz multiplicity of $R_{3,d}$ was already computed in \cite{yos}, which he computes using Corollary \ref{cor: yoshida_theorem_mu_1}.
    
    To compute $\eHK{S_{3,d}}$, we will use basically the same idea that Yoshida uses to compute $\eHK{R_{2,d}}$ for even–dimensional $d$: first note that $k[[x_0,\dots,x_{d-1},x_d^3]]\subset k[[x_0,\dots,x_d]]$ is a free extension of rank $3$. Thus, the following chain of inequalities follows:
    \begin{align*}
    \HK_e(S_{3,d}) &=\dim_k \frac{k[[x_0,\dots,x_d]]}{(x_0^{3^e},\dots,x_d^{3^e})+(x_0^2+\dots+x_{d-1}^2+x_d^3)}\\
    &= 3\cdot \dim_k \frac{k[[x_0,\dots,x_d^3]]}{(x_0^{3^e},\dots,x_d^{3^e})+(x_0^2+\dots+x_{d-1}^2+x_d^3)}\\
    &= 3\cdot \dim_k \frac{k[[x_0,\dots,x_{d-1}]]}{(x_0^{3^e},\dots,x_{d-1}^{3^e},(x_0^2+\dots+x_{d-1}^2)^{3^{e-1}})},
    \intertext{and now, since $k[[x_0^{3^{e-1}},\dots,x_{d-1}^{3^{e-1}}]]\subset k[[x_0,\dots,x_{d-1}]]$ is also a free extension of rank $3^{d(e-1)}$,}
    \HK_e(S_{3,d}) &=3\cdot 3^{(e-1)d}\cdot \dim_k \frac{k[[x_0^{3^{e-1}},\dots,x_{d-1}^{3^{e-1}}]]}{(x_0^{3^e},\dots,(x_0^2+\dots+x_{d-1}^2)^{3^{e-1}})}\\
    &= 3\cdot 3^{(e-1)d}\cdot \dim_k \frac{k[[y_0,\dots,y_{d-1}]]}{(y_0^3,\dots,y_{d-1}^3)+(y_0^2+\dots+y_{d-1}^2)} = 3\cdot 3^{(e-1)d}\cdot \HK_1(R_{3,d-1}).
    \end{align*}
    In the mentioned Yoshida's notes, it is proven that $\HK_1(R_{3,d-1}) = 2^{d-1}+3^{d-1}$, so
    \begin{gather*}
    \eHK{S_{3,d}} = \frac{3\cdot 3^{d(e-1)}\cdot (2^{d-1}+3^{d-1})}{3^{de}} = 1+\left(\frac{2}{3}\right)^{d-1}
    \end{gather*}

\end{proof}

\begin{remark}
In general, one can use the same trick to compute the Hilbert–Kunz function of the $A_{p^e-1}$ singularity in charactersitic $p$: one obtains that the Hilbert–Kunz multiplicity of $\mathbb{F}_p[[x_0,\dots,x_d]]/(x_0^2+\dots+x_{d-1}^2+x_d^{p^e})$ at the maximal ideal is $\frac{\HK_e(R_{p,d-1})}{p^{(d-1)e}}$. Even further, the Hilbert–Kunz multiplicity at the maximal ideal of $\mathbb{F}_p[[x_0,\dots,x_d]]/(x_0^{e_0}+\dots+x_{d-1}^{e_{d-1}}+x_d^{p^e})$ would be $\HK_e\left(\mathbb{F}_p[[x_0,\dots,x_{d-1}]]/(x_0^{e_0}+\dots+x_{d-1}^{e_{d-1}})\right)/p^{(d-1)e}$.
\end{remark}

\subsection{The main inequality, \texorpdfstring{$p=2,3$}{p23}}\label{subsec: main_inequality_char_2_and_3}
We can now settle the problem in characteristics 2 and 3. First, let us show that the computations above show that the Hilbert–Kunz multiplicity of the $A_1$ and $A_2$ singularities are different: 
\begin{theorem}[Main inequality in characteristic 2]
    \label{th: A_1_A_2_ineq_char_2}
    For $d\geq 2$, $\eHK{S_{2,d}}>\eHK{R_{2,d}}$.
\end{theorem}
\begin{proof}
Let $d=2m-1$. Then, from Theorems \ref{th: HK_function_A_1_char_2} and \ref{th: HK_function_A_2_char_2},
\begin{gather*}
    \eHK{R_{2,d}}<\eHK{S_{2,d}} \Leftrightarrow \frac{2^m}{2^m-1}<\frac{2^{m-1}+1}{2^{m-1}} \Leftrightarrow 2^d<2^d+2^{m-1}-1,
\end{gather*}
which is true for every $m>1$. If $d = 2m$, from \textit{loc. cit.}
\begin{gather*}
    \eHK{R_{2,d}}<\eHK{S_{2,d}} \Leftrightarrow \frac{2^m+1}{2^m}<\frac{2^{m+1}+1}{2^{m+1}-1} \Leftrightarrow 2\cdot 2^d + 2^m-1<2\cdot 2^d+2^{m},
\end{gather*}
which is again true for every $m\geq 1$.
\end{proof}

\begin{theorem}[Main inequality in characteristic 3]
    \label{th: A_1_A_2_ineq_char_3}
    For $d\geq 2$, $\eHK{S_{3,d}}>\eHK{R_{3,d}}$.
\end{theorem}
\begin{proof}
By Theorem \ref{th: eHK_A_1_and_A_2_in_char_3}, we have to compare the terms $(2/3)^{d-1}$ and $3\cdot 2^d/(3^{d+1}-2^d+(-1)^d)$. Now,
    \begin{gather*}
        (2/3)^{d-1}> 3\cdot 2^d/(3^{d+1}-2^d+(-1)^d)\Leftrightarrow 3^{d+1}-2^d+(-1)^d > 3^d\cdot 2 \Leftrightarrow 3^d-2^d+(-1)^d >0,
    \end{gather*}
    where the second and last equivalence hold because $3^d-2^d+(-1)^d>0$ for all $d\geq 2$.
\end{proof}

We finish the section with a brief observation. A corollary to the formula of $\eHK{S_{3,d}}$ and $\eHK{R_{2,d}}$ is that Conjecture 2.7 in \cite{jnskw} is only true up to dimension $7$ for complete intersections, and false in general, $\eHK{S_{3,d}}$ being a counterexample:
\begin{corollary}
    \label{cor: better_char_free_lower_bound}
    $\eHK{S_{3,d}}<\eHK{R_{2,d}}$ if and only if $d\geq 8$ and even or $d\geq 13$ and odd.
\end{corollary}
\begin{proof}
    It follows from direct computation with the formula in Theorems \ref{th: eHK_A_1_and_A_2_in_char_3} and \ref{th: HK_function_A_1_char_2}.
\end{proof}
\begin{remark}
This shows that $\eHK{R_{2,d}}$ is a characteristic-free lower bound for Hilbert–Kunz multiplicity in dimension up to $7$ that is sharper than $\lim_{p\rightarrow \infty} \eHK{R_{p,d}}$.
\end{remark}
\section{The Strong Watanabe–Yoshida conjecture for complete intersections}\label{sec: wyc_revisited}
The inequalities proven in the previous sections already settle the Strong Watanabe–Yoshida conjecture for hypersurfaces in characteristic $p>3$. As was already mentioned in the introduction, obtaining the result for complete intersections in every charactersitic requires refining Enescu and Shimomoto's Theorem \ref{th: enescu_shimomoto_intro}. 

In §\ref{subsec: usc_and_wp}, we recall the two main results we need for this. First, in §\ref{subsec:enescu_shimomoto_char_2_3}, we will prove the result for hypersurfaces in the remaining characteristics 2 and 3 cases, and next in §\ref{subsec:shorter_and_stronger_proof}, we generalize the statement from hypersurfaces to complete intersections. In addition, this result provides a shorter proof to Theorem \ref{th: enescu_shimomoto_intro}.
\subsection{Dense upper semi-continuity and Weierstrass' Preparation Lemma}\label{subsec: usc_and_wp}
First, we recall a dense upper semi-continuity result by Enescu and Shimomoto, which we will use throughout the section.
\begin{theorem}[Enescu–Shimomoto, \cite{enshi}, Proposition 4.2]
    \label{th: usc}
    Let $k$ be an uncountably infinite field of characteristic $p>0$. Let $A= k[[x_1, ..., x_n]]$ and
    $R := A/(f_1,... ,f_\ell)$ a complete intersection ring and $f, g \in A$ such that they form a regular sequence in $R$. Let $0 \neq h \in R$. Then there
    exist a dense subset $V \subset k$ such that $ah + f, g$ form a regular sequence on $R$ and
    \begin{gather*}
    \eHK{\frac{R}{(f,g)}}\geq \eHK{\frac{R}{(ah+f,g)}}
    \end{gather*}
    for all $a \in V$.
\end{theorem}

We now recall the Weierstrass preparation lemma for reference. We will use it in the final part of this section for $(A,\m) = k[[x_0,\dots,x_{d+1}]]$ and $(B,\n) = k[[x_0,\dots,x_{d}]]$, which are both complete Noetherian local rings.
\begin{definition}
Given $(B,\n)$ a complete local Noetherian ring, we say that $\varphi\in B[[x]] = A$ is a distinguished polynomial on the variable $x$ if for some $n$,
\begin{gather*}
\varphi = x^n+b_1x^{n-1}+\dots+b_{n-1}x+b_n,
\end{gather*}
where $b_i \in \n\subset B$.
\end{definition}
\begin{theorem}[Weierstrass' Preparation Lemma, \cite{ge83}]
    \label{th: wp}
    Let $(B,\n)$ be a complete local Noetherian ring. Let $F = \sum b_ix^i\in B[[x]]$ where $b_i\in B$. If $b_i \in \m$ for all $i<n$ and $b_n\notin \m$, then
    \begin{gather*}
    F = u\varphi,
    \end{gather*}
    where $u\in B[[x]]$ is a unit, and $\varphi$ is a distinguished polynomial on $x$ of degree $n$.
\end{theorem}

\subsection{The hypersurface case in characteristics 2 and 3}\label{subsec:enescu_shimomoto_char_2_3}

In this section, we will extend Theorem 3.4 in \cite{enshi} to characteristics 2 and 3, which, together with Theorems \ref{th: A_1_A_2_ineq_char_2} and \ref{th: A_1_A_2_ineq_char_3}, yield the Strong Watanabe–Yoshida Conjecture for complete intersections in these two cases. Enescu and Shimomoto's argument to settle this problem in characteristic $p>2$ relies on Lemma 3.1 in \cite{enshi}, which is a lemma to \say{complete squares} in a power series ring over a field whose characteristic is not 2. However, that lemma does not work in characteristic 2. Nevertheless, in this section, we prove an alternative method for completing squares in characteristic 2 that addresses this issue.

\begin{notation}
    We will use multi-index notation in the following lemmas, so we will write $f = \sum f_\ue x_1^{e_1}\dots x_n^{e_n}$ where $\ue:=(e_1,\dots,e_n)$, and we define $|\ue| = \sum e_i$ the norm of the multi-index.
\end{notation}

\begin{lemma}
\label{lm: quadratic_equation_power_series_char_2}\label{lm: units_cubic_roots_in_char2}
    Let $k$ be an algebraically closed field of characteristic 2 and $v,w\in k[[x_0,\dots,x_d]]$, where $v$ is a unit. Then,
    \begin{enumerate}[label=(\alph*)]
    \item There exists $t\in k[[x_0,\dots,x_d]]$ such that $t^2+vt+w = 0$. Moreover, if $w$ is a unit, then $t$ also is.
    \item There exists a unit $s\in k[[x_0,\dots,x_d]]$ such that $s^3 = v$.
    \end{enumerate}
\end{lemma}
\begin{proof}
    We first prove (a). Let $u=v^{-1}$. A solution to the equation exists if and only if $t$ exists such that $t = v^{-1}(t^2+w) = u(w+t^2)$, which is the case if and only if the following set of equations has a solution over $k$:
    \begin{gather}
        \label{fr: coefficients_of_t}
        t_{\ue} = \sum_{\ue'+\ue'' = \ue} u_{\ue'}(w_{\ue''}+\sum_{\uf+\uf'=\ue''} t_{\uf}t_{\uf'}).
    \end{gather}
    Observe that $t_\ue$ only appears on the RHS only if $\uf = 0, \uf' = \ue$ or $\uf = \ue$ and $\uf' = 0$. In both these cases, we have that $t_{\uf}t_{\uf'} = t_0t_{\ue}.$ Because $k$ has characteristic 2, these two terms cancel each other out. Thus, for every $\ue$, the expression on the right-hand side depends on coefficients of $t$ of monomials of strictly lower degree, i.e., it depends on coefficients $t_{\uf},t_{\uf'}$ such that $|\uf|,|\uf'|< |\ue|.$

    The equation for $\ue = 0$ has a solution because the field is algebraically closed, so it follows by induction on the norm of the multi-index that the set of equations (\ref{fr: coefficients_of_t}) has a solution.

    Let us now prove (b). Let $u = \sum b_{\underline{e}} x_0^{e_0}\dots x_d^{e_d}$. There exists a series $s = \sum a_{\underline{e}} x_0^{e_0}\dots x_d^{e_d}$ such that $s^3 = u$ if and only if the following systems of equations have a solution:
\begin{gather}
    \label{fr: equations_cubic_root_char_2}
    \sum_{\underline{i}+\underline{j}+\underline{k} = \ue} a_{\underline{i}}a_{\underline{j}}a_{\underline{k}} = b_{\ue}, \ue \in \mathbb{Z}^{d+1}_{\geq 0}.
\end{gather}
The fact that these systems of equations have a solution in $k$ follows from induction on the norm of the multi-index $\ue$. 
\end{proof}

The following lemma is the characteristic 2 analog of Lemma 3.1 in \cite{enshi}, used for completing squares in characteristic not 2.
\begin{lemma}[Completing squares in characteristic 2]
\label{lm: power_series_completing_squares_in_char_2}
    Let $k$ be an algebraically closed field of characteristic 2 and put $A = k[[x_2,\dots,x_d]]$. Consider $B = A[[x_0,x_1]]$ and $F = x_0x_1 + Q(x_2,\dots,x_d) + G$ with $Q\in A$ a quadratic form and $G\in \mathfrak{m}_B^3$, where $\mathfrak{m}_B$ is the maximal ideal of $B$. Then there exist $a_0,a_1,b_0,b_1 \in B$ units, $a_2,b_2\in (x_2,\dots,x_d)B$ and $G_1\in (x_2,\dots,x_d)^3B$ such that 
    \begin{gather*}
        F = (a_0x_0+a_1x_1+a_2)(b_0x_0+b_1x_1+b_2)+Q(x_2,\dots,x_d) + G_1.
    \end{gather*}
\end{lemma}
\begin{proof}
    Let $x = x_0$ and $y = x_1$ for simplicity. We can assume that $F = x^2+xy+y^2+Q_{d-2}(x_2,\dots,x_d)$ after a change of variables. Then, note that we can express $G$ in the following way:
\begin{equation}
\label{gee}
G = c + ax + by + Ax^2 + Bxy + Cy^2,
\end{equation}
where $c\in (x_2,\dots, x_d)^3,$ $a,b\in (x_2,\dots, x_d)^2,$ $A,B,C\in (x_2,\dots,x_d).$

Let $u := 1 + A$, $v := 1 + B$ and $w := 1 + C$, which are units. Due to Lemma \ref{lm: quadratic_equation_power_series_char_2}, there exists a unit $t$ such that $ut^2+vt+w = 0$. The result follows by letting 
\begin{gather*}
\begin{array}{lll}
    a_0 = u,& a_1 = t^{-1}w,&a_2=uv^{-1}b+t^{-1}wv^{-1}a,\\
    b_0 = 1,& b_1 = t,& b_2 = v^{-1}ta+v^{-1}b,
\end{array} 
\end{gather*}
and $G_1 = c+a_2b_2$.
\end{proof}
\begin{lemma}
    \label{lm: change_of_variables_xy}
    Using the notation above, there is a change of variables transforming $(a_0x+a_1y+a_2)(b_0x+b_1y+b_2)$ into $xy.$
\end{lemma}
\begin{proof}
    Let us define a change of variables where $x_i\mapsto x_i$ for $i\geq 2$, and 
    \begin{gather*}
        \left(\begin{array}{c} x\\ y\end{array}\right)\mapsto \left(\begin{array}{cc} a_0 & a_1\\ b_0 & b_1\end{array}\right)^{-1}\left(\begin{array}{c} x\\ y\end{array}\right).
    \end{gather*}
    This matrix is invertible because its determinant is $a_0b_1+a_1b_0 = v,$ which is a unit.
\end{proof}
We will also use the following lemma in the next result.
\begin{corollary}[Structure theorem for quadratic forms in characteristic 2, Theorem 7.31, \cite{elkame}]
    \label{cor: diagonalization_quadratic_forms_char_2}
        Let $k$ be an algebraically closed field of characteristic 2. Let $Q$ be a quadratic form over $k$. Then, $Q = Q_r$ for some $r$, for $Q_r$ as in Definition \ref{df: Qd}.
    \end{corollary}

The following result is the extension of Theorem 3.4 in \cite{enshi} to characteristic 2.
\begin{theorem}
    \label{th: enescu_shimomoto_char_2}
    Let $(R,\m,k)$ be a non-regular hypersurface ring of characteristic 2. Then, either $\widehat{R}\cong R_{2,d}\widehat{\otimes}_{\mathbb{F}_2}k,$ or
        \begin{gather*}
        \eHK{R}\geq \eHK{S_{p,d}}> \eHK{R_{p,d}}.
        \end{gather*}
    \end{theorem}
    \begin{proof}
        Passing to the completion and extending the residue field to an uncountably infinite algebraically closed field does not change Hilbert–Kunz multiplicity. Thus, we can assume that $R = k[[x_0,...,x_d]]/(F),$ and we will write $F = Q+G$ where $Q$ is the homogeneous part of degree $2$ and $G\in \mathfrak{m}^3$. 
        
        Now, the fact that $R\ncong R_{d,2}\widehat{\otimes}_{\mathbb{F}_p} k$ implies that $Q$ is a degenerate quadric. Indeed, assume that $Q$ is a nondegenerate quadric. Lemma \ref{cor: diagonalization_quadratic_forms_char_2} states that after a linear change of variables, we have that $Q = Q_d$. Then, by successive application of Lemmas \ref{lm: power_series_completing_squares_in_char_2} and \ref{lm: change_of_variables_xy}, one can transform the entire power series $F$ into the quadric $Q_d$ via a series of isomorphisms. Thus, $R\cong R_{d,2}\widehat{\otimes}_{\mathbb{F}_p} k$.
        
        Hence, thanks to Lemma \ref{cor: diagonalization_quadratic_forms_char_2}, we can assume, after a suitable change of variables, that $Q = Q_r$ for some $r<d$. Let us now define the polynomial with which we will deform the hypersurface. Let $\delta = 0$ if the coefficient of $x_d^3$ in $G$ is non-zero, and $\delta = 1$, otherwise.
        \begin{enumerate}
            \item If $r$ odd, then define 
            \begin{equation*}
                g := \begin{cases}
                    x_{r+1}x_{r+2} + \dots + x_{d-1}^2 + \delta x_d^3 & \text{if $d$ is odd},\\
                    x_{r+1}x_{r+2} + \dots + x_{d-2}x_{d-1} + \delta x_d^3 & \text{if $d$ is even}.
                \end{cases}
            \end{equation*}
            \item If $r$ is even, then let
                \begin{equation*}
                    g = \begin{cases}
                        x_{r}x_{r+1} + \dots + x_{d-1}^2 + \delta x_d^3 & \text{if $d$ is odd},\\
                        x_{r}x_{r+1} + \dots + x_{d-2}x_{d-1} + \delta x_d^3 & \text{if $d$ is even}.
                    \end{cases}
                \end{equation*}
        \end{enumerate}
            Due to the dense upper semi-continuity of the Hilbert-Kunz multiplicity, Theorem 2.6 in \cite{enshi}, there exists a dense $V\subset k$ such that $\forall a\in \Lambda\setminus \lbrace 0 \rbrace,$ then
            \begin{equation*}
                e_{HK}\left(k[[x_0,\dots,x_d]]/(F+a g)\right) \leq e_{HK}\left(k[[x_0,\dots,x_d]]/(F)\right).
            \end{equation*} 
            We make a different transformation depending on the parity of $r$:
        \begin{enumerate}
            \item If $r$ is odd, then $F_a = F+a g$ takes the following expression
            \begin{equation*}
                x_0x_1 + \dots + x_{r-1}x_r + a x_{r+1}x_{r+2} + \dots + \begin{cases}
                a x_{d-1}^2 + a \delta x_d^3 + G & \text{if $d$ is odd},\\
                a \delta x_d^3 + G & \text{if $d$ is even}.
                \end{cases}
            \end{equation*}
            Rescaling the coordinates, we can assume that $a = 1$, $\delta = 1$, and $G \in \m^3$. 
            \item If $r$ is even, then the same argument we used above yields that
                \begin{equation*}
                    F_a = x_0x_1 + \dots + x_r^2 + a x_{r}x_{r+1} + \dots +
                    \begin{cases}
                    a x_{d-1}^2 + a \delta x_d^3 + G & \text{if $d$ is odd}\\
                    a \delta x_d^3 + G &\text{if $d$ is even}
                    \end{cases}
                \end{equation*}
                where we can again assume $a = 1$, $\delta = 1$ and $G \in \m^3$. Then transform $F_1$ via the change of variables defined by $x_{r} + x_{r+1}\mapsto x_{r+1}$ and $x_i\mapsto x_i$ for $i\neq r+1$ to obtain
            \begin{equation*}
            F_1 = x_0x_1 + \dots + x_{r}x_{r+1} + \dots +
            \begin{cases}
            x_{d-1}^2 + x_d^3 + G & \text{if $d$ is odd,}\\
            x_d^3 + G &\text{if $d$ is even.}
            \end{cases}
            \end{equation*}
        \end{enumerate}
        Now, if $d\geq 2$, Lemmas \ref{lm: power_series_completing_squares_in_char_2} and \ref{lm: change_of_variables_xy} imply that $F_a = x_0x_1 + P_{d-2} + G_1$ where $G_1\in (x_2,\dots,x_d)^3$. If $d-2\geq 2$, we can apply the lemmas again to $P_{d-2}+G_1$, and continue recursively until we obtain the following:
        \begin{equation*}
            F_1 = x_0x_1 + \dots + \begin{cases} x_{d-1}^2 + x_d^3 + G & \text{if } d \text{ odd, where } G\in (x_{d-1}^3,x_{d-1}^2x_d,x_{d-1}x_d^2,x_d^4), \\ x_d^3 + G & \text{if } d \text{ even, where } G\in (x_d)^4.\end{cases}
        \end{equation*}
        Let's study both cases separately.
        \begin{enumerate}
        \item Let $d$ be odd. The general case is easily deduced from the case $d = 1$, so let $x = x_0$ and $y = x_1$ for simplicity. Note that $G = Ax^3 + Bx^2y + Cxy^2 + Dy^3$ where $A,B,C,D\in k[[x,y]],$ and $D$ is not a unit, so
        \begin{align*}
        F_1 =&\ x^2 + y^3 + Ax^3 + Bx^2y + Cxy^2 + Dy^3\\
        =&\ (1+Ax+By)x^2 + ((1+D)y+Cx)y^2.
        \end{align*}
        Let $u := 1+D$, which is a unit. Doing the change of variables $x \mapsto x; y \mapsto u^{-1}(Cx+y),$
        \begin{gather*}
            (1+Ax+Bu^{-1}(Cx+y))x^2+yu^{-2}(Cx+y)^2 = vx^2+u^{-2}y^3,
        \end{gather*}
        where $v = 1+(A+Bu^{-1}C)x+(Bu^{-1}+C^2u^{-2})y$ is a unit. Now, observe that $F_1$ and $v^{-1}F_1$ define the same principal ideal, so we can assume that $F_1 = x^2+wx^3$, with $w = (vu^2)^{-1}$. By Lemma \ref{lm: units_cubic_roots_in_char2}, there exists a unit $t$ such that $t^3 = w$, so after applying the change of variables $x \mapsto x; y \mapsto t^{-1}y,$
        we finally map $F_1$ to $x^2+y^3.$
        \item If $d$ is even, we can then write $F_1 = x_0x_1+\dots+x_{d-2}x_{d-1}+x_d^3(1+H)$ where $H$ is not a unit. Let $w = 1+H$. Since $F_1$ and $w^{-1}F_1$ generate the same principal ideal, we can assume that $F_1 = w^{-1}x_0x_1+\dots+w^{-1}x_{d-2}x_{d-1}+x_{d}^3$. Finally, we can apply the change of variables
        \begin{gather*}
        \begin{cases}
            x_{2m} \mapsto w^{-1}x_{2m}, &m<d/2\\
            x_{2m+1} \mapsto x_{2m+1}, &m<d/2\\
            x_d \mapsto x_d,
        \end{cases}
        \end{gather*}
        which maps $F_1$ to $x_0x_1+\dots+x_{d-2}x_{d-1}+x_d^3.$
        \end{enumerate}
    To sum up, using dense upper semi-continuity, we have shown that if $R\ncong R_{p,d}\widehat{\otimes}_{\mathbb{F}_p} k$, we can deform the hypersurface $F$ down to an $A_2$ singularity by transformations that do not increase Hilbert–Kunz multiplicity.
    \end{proof}
Finally, with a slight modification of Enescu and Shimomoto's original argument, we prove the theorem above also in characteristic 3.
\begin{theorem}
    \label{th: enescu_shimomoto_char_3}
    Let $(R,\m,k)$ be a non-regular hypersurface ring of characteristic $3$. Then, either $\widehat{R}\cong R_{3,d}\widehat{\otimes}_{\mathbb{F}_3} k,$ or
    \begin{gather*}
    \eHK{R}\geq \eHK{S_{3,d}}> \eHK{R_{3,d}}.
    \end{gather*}
\end{theorem}
\begin{proof}
    We can assume that $R$ is complete and that the residue field is uncountably infinite and algebraically closed, so let $R = k[[x_0,\dots,x_d]]/(F)$. If we assume that $R \ncong R_{3,d}\widehat{\otimes}_{\mathbb{F}_3} k$, then using \ref{th: usc}, we can deform $F$ to $F+ag = F_a = v_0x_0^2 + \dots + v_{d-1}x^2_{d-1} + v_dx_d^3,$ where the $v_i$ are units, and the Hilbert-Kunz multiplicity of $k[[x_0,\dots,x_d]]/(F_a)$ is less than that of $R$. Since $F_a$ and $v_d^{-1}F_a$ generate the same ideal, we can assume that $v_d = 1$. We can now apply Lemma 3.3 in \cite{enshi} to solve the equations $w_i^2 = v_i,$ for $i<d.$ Then, $F_a = x_0^2 + \dots + x_{d-1}^2 + x_d^3,$ and the first inequality follows.
\end{proof}

    \subsection{The complete intersection case}\label{subsec:shorter_and_stronger_proof}
        Let us now prove that the Strong Watanabe–Yoshida conjecture for complete intersections can be reduced to the hypersurface case. This is a generalization of Theorem \ref{th: enescu_shimomoto_intro} to complete intersections of every positive characteristic and also a shorter proof. First, we provide another lemma about quadratic forms.
    \begin{lemma}
    Let $k$ be an algebraically closed field of any characteristic. Let $n\geq 2$ and $Q\in k[[x_1,\dots,x_n]]$ be a nondegenerate quadratic form. Then, there is a linear change of variables such that
    \begin{gather*}
    Q \mapsto Q'(x_1,\dots,x_{n-2})+x_{n-1}x_n.
    \end{gather*}
    \end{lemma}
    \begin{proof}
    See \cite{elkame}. This is a consequence of the structure theorems for quadratic forms: Proposition 7.29 in characteristic not 2, and Proposition 7.31 in characteristic 2.
    \end{proof}
    \begin{theorem}
        \label{th: reduction_to_hypersurface}
        Let $(R,\m,k)$ be a non-regular complete intersection such that $\widehat{R}\not\cong R_{p,d}\widehat{\otimes}_{\mathbb{F}_p} k$. Then, there exists a non-regular hypersurface $F$ such that
        \begin{gather*}
            \eHK{R}\geq \eHK{k[[x_0, ..., x_{d}]]/(F)}.
        \end{gather*}
        Furthermore, it can be shown that the hypersurface is not an $A_1$ singularity, i.e. 
        \begin{gather*}
            \frac{k[[x_0, ..., x_{d}]]}{(F)}\not\cong R_{p,d}\widehat{\otimes}_{\mathbb{F}_p} k.
        \end{gather*}
    \end{theorem}
    \begin{proof}
        We can assume that $R$ is complete and extend the residue field to make it uncountably infinite and algebraically closed without loss of generality. Note that the case of hypersurfaces is already settled. Hence, first, let us see that the problem for the rest of singular complete intersections reduces to codimension $2$. Thus, say
        \begin{gather*}
        R = \frac{k[[x_1,\dots,x_n]]}{(f_1,\dots,f_\ell)},
        \end{gather*}
        with $\ell\geq 3$, where $f_1,\dots,f_\ell\in (x_1,\dots,x_n)$ form a regular sequence. Note that it can be further assumed that $f_i\in \m^2$ for all $i$: indeed, if an element, say $f_\ell$, defines a regular hypersurface, then $R \stackrel{\varphi}\cong k[[y_1,\dots,y_{n-1}]]/(f_1',\dots,f_{\ell-1}')$ where if $f_i\in \m^2$, then $f_i':=\varphi(f_i)\in (y_1,\dots,y_{n-1})^2$.

        Now, by Theorem \ref{th: usc}, there exists a Zariski-dense set $V\subset k$ such that for every $a\in V$, the elements $f_1,\dots,f_\ell+ax_n$ form a regular sequence in $k[[x_1,\dots,x_n]]$, and 
        \begin{gather*}
        \eHK{R}\geq \eHK{\frac{k[[x_1,\dots,x_n]]}{(f_1,\dots,f_\ell + ax_n)}}.
        \end{gather*}
        In this new regular sequence, the element $f_\ell+ax_n$ defines a regular hypersurface, and therefore the same argument at the beginning shows that we can eliminate one variable and one equation thus obtaining another non-regular complete intersection of the same dimension but with one equation less:
        \begin{gather*}
            \frac{k[[x_1,\dots,x_n]]}{(f_1,\dots,f_\ell + ax_n)} \cong \frac{k[[y_1,\dots,y_{n-1}]]}{(f_1',\dots,f_{\ell-1}')},
        \end{gather*}
        and moreover, $f_i'\in (y_1,\dots,y_{n-1})^2$. Repeating this process while $\ell\geq 3$ yields that we can reduce the statement to singular complete intersections of codimension 2, that is,
    \begin{gather*}
    R = \frac{k[[x_0,\dots,x_{d+1}]]}{(f,g)}, \text{ where } f,g\in \m^2.
    \end{gather*}
Let us write $g = Q+G$ where $Q$ is the homogeneous part of degree $2$ of $g$ and $G\in \m^3$. For the theorem, we can assume that $Q$ is a nondegenerate quadric using dense upper semi-continuity: by Theorem \ref{th: usc}, there exists a Zariski-dense set $W\subset k$ such that for all $a\in W$, 
\begin{gather*}
    \eHK{R}\geq \eHK{\frac{k[[x_0,\dots,x_{d+1}]]}{(f,g+aQ_{d+1})}}.
\end{gather*}
Since there are only finitely many values in $k$ that make $g+aQ_{d+1}$ a degenerate quadric, we can choose $a\in W$ in such a way that it is nondegenerate.
    
    Hence, by the lemma above and after a linear change of variables, we can further assume that $Q = Q'(x_0,\dots,x_{d-1})+x_{d}x_{d+1}$.
    
    Now, again by the dense upper semi-continuity result \ref{th: usc}, there exists a Zariski-dense set $V\subset k$ such that for every $a\in V$, $f+ax_{d+1},g$ form a regular sequence, and
    \begin{gather*}
    \eHK{\frac{k[[x_0,\dots,x_{d+1}]]}{(f,g)}}\geq \eHK{\frac{A}{(f+ax_{d+1},g)}}.
    \end{gather*}
    After a linear change of variables, provided we choose $a\in V\setminus\lbrace 0\rbrace$, we can assume that $a=1$.

    Now, we will apply the Weierstrass preparation lemma to conclude our result. Let us write $x := x_{d+1}$ for simplicity, and $f = b_0+b_1x+b_2x^2+\dots$ where $b_i \in k[[x_0,\dots,x_d]]$. Since $f\in \m^2$, we have that 
    \begin{gather*}
        b_0,b_1 \in (x_0,\dots,x_{d})k[[x_0,\dots,x_{d}]]
    \end{gather*}
    which implies that $1+b_1$ is a unit in $k[[x_0,\dots,x_{d}]]$. Hence, the Weierstrass preparation lemma states that 
    \begin{gather}
        \label{fm: WP}
        x+f = u(x-b)
    \end{gather}
    where $u\in k[[x_0,\dots,x_{d+1}]]$ is a unit and $b\in (x_0,\dots,x_{d})k[[x_0,\dots,x_{d}]]$. 
    
    However, it can be proven that $b\in ((x_0,\dots,x_{d}) k[[x_0,\dots,x_{d}]])^2$. Let us denote by $[P]_i$ the homogeneous part of degree $i$ of a power series $P$. We already know that $[u]_0\neq 0$, $[b]_0 = 0$ and $[f]_0=[f]_1=0$ since $u$ is a unit, $b\in (x_0,\dots,x_{d})k[[x_0,\dots,x_{d}]]$ and $f\in \m^2$. So we only need to check that $[b]_1 = 0$. If we write out the degree 1 part of both sides of (\ref{fm: WP}), we get
    \begin{gather*}
        \left[x+f\right]_1 = x\\
        \left[u(x-b)\right]_1 = [u]_0x+[u]_0[b]_1
    \end{gather*}
    so $[u]_0x+[u]_0[b]_1 = x$. But $[b]_1\in k[[x_0,\dots,x_d]]$ cannot be a multiple of $x$, so $[b]_1 = 0$.

    We can finally obtain the result: since $(f+x,g) = (x-b,g)$ where $b \in ((x_0,\dots,x_{d}) k[[x_0,\dots,x_{d}]])^2$, then
    \begin{gather*}
    \frac{k[[x_0,\dots,x_{d+1}]]}{(f+x_{d+1},g)} \cong \frac{k[[x_0,\dots,x_{d}]]}{(g(x_0,x_2,\dots,x_{d},b(x_0,\dots,x_{d})))}
    \end{gather*}
    Now, note that $F = g(x_0,x_2,\dots,x_{d},b(x_0,\dots,x_{d}))$ is the hypersurface we were looking for, because 
    \begin{gather*}
        g(x_0,\dots,x_{d},b) = Q'(x_0,\dots,x_{d-1})+x_{d}b + G(x_0,\dots,x_{d},b),
    \end{gather*}
    where $x_{d}b\in (x_0,\dots,x_{d})^3$, and therefore the homogeneous part of degree two of $F$ is the quadric $Q'$, which cannot be transformed into $Q_d$. The series $F$ thus cannot define an $A_1$ singularity $R_{p,d}\widehat{\otimes}_{\mathbb{F}_p} k$.
    \end{proof}
    \begin{corollary}
        \label{maintheoremA}
        Let $(R,\m,k)$ be a non-regular complete intersection of characteristic $p>0$ and dimension $d\geq 2$. Then, either $\widehat{R}\cong R_{p,d}\widehat{\otimes}_{\mathbb{F}_p}k$, or
        \begin{gather}
        \eHK{R}\geq \eHK{S_{p,d}}> \eHK{R_{p,d}}.
        \end{gather}
        In particular, the Strong Watanabe–Yoshida conjecture holds for complete intersections in all positive characteristics.
    \end{corollary}
    \begin{proof}
    If $R$ is a complete intersection of codimension at least $2$, by Theorem \ref{th: reduction_to_hypersurface}, there exists a non-regular hypersurface $R' = k[[x_0,\dots,x_d]]/(F)$ of the same dimension such that $\eHK{R}\geq \eHK{R'}$, and such that $R'\neq R_{p,d}\widehat{\otimes}_{\mathbb{F}_p}k$.
    
    Then, the following inequalities follow from Theorem \ref{th: enescu_shimomoto_intro} for $p>3$, and from \ref{th: enescu_shimomoto_char_2} and \ref{th: enescu_shimomoto_char_3} for characteristics $2$ and $3$, respectively:
    \begin{gather*}
    \eHK{R'}\geq \eHK{S_{p,d}}\geq \eHK{R_{p,d}}.
    \end{gather*}
    Finally, that the second inequality is strict is due to Theorems \ref{cr: A_2_A_1_ineq_large_char} for $p>3$, \ref{th: A_1_A_2_ineq_char_2} and \ref{th: A_1_A_2_ineq_char_3} for characteristics $2$ and $3$, respectively.
    \end{proof}

    \begin{remarks}
        In the proof of Theorem \ref{th: reduction_to_hypersurface}, we start assuming that the ring has codimension 2, but leveraging on some inequalities from \cite{aben08}, we could further assume that the ring $R$ is a strongly $F$-regular ring with $\eHK{R}<\frac{4}{3}$. Computer experiments suggest that it is unlikely that a non-regular complete intersection of codimension 2 attains such a small Hilbert–Kunz multiplicity. This observation raises the question of whether there is a more direct way to show this result without relying on dense upper semi-continuity. Ian Aberbach informed me that he and Cătălin Ciupercă had indeed obtained $\eHK{R}\geq \eHK{R_{p,d}}$ for singular complete intersections, without relying on dense upper semi-continuity.

        It is also reasonable to wonder whether the stronger lower bound for complete intersections that Corollary \ref{maintheoremA} provides also works for non-complete intersections, i.e., whether an unmixed non-regular ring $R$ could exist such that $\eHK{S_{p,d}}> \eHK{R}> \eHK{R_{p,d}}$.
    \end{remarks}
\section{Revisiting low dimension and low multiplicity}\label{sec:revisiting_literature}

As mentioned in the introduction, the first part of the Watanabe–Yoshida conjecture \ref{watanabe_yoshida} has also been settled up to dimension $7$, and also for rings of Hilbert–Samuel multiplicity up to 5, as long as $p\neq 2$. The only instances where the strong part of the conjecture has also been settled are dimensions 2 through 4, provided that the characteristic is not 2. In dimensions 5 through 7, or in the low multiplicity cases, the strong part of the conjecture was not considered and did not always follow, since some of the arguments depended on the reduction to complete intersections, or the lower bounds were not strong enough for characteristic 2.

Furthermore, the characteristic 2 case is hardly ever covered in the literature; the proof of the two-dimensional case is characteristic-free, but it was not included as part of the general statement of the Watanabe–Yoshida conjecture until \cite{yos}.

The goal of this section is to fill these gaps in the literature: we will extend these low-dimensional (and low-multiplicity) cases of the conjecture to characteristic 2, and also show that the strong version of the conjecture follows. In other words, the main result of this section is the following:
\begin{theorem}[Watanabe–Yoshida conjecture]
    \label{th: watanabe_yoshida_low_dimension}
    Let $p>0$ be prime and $2\leq d \leq 6$, or $p>2$ and $d=7$. Set
    \begin{equation*}
        R_{p,d}:=\mathbb{F}_p[[x_0,...,x_d]]/(Q_d),
    \end{equation*}
    where
    \begin{equation*}
        Q_d = \begin{cases}x_0x_1+x_2x_3 + \dots + x_{d-1}x_d & \text{if $d$ is odd},\\
        x_0x_1 + x_2x_3 + \dots +x_{d-2}x_{d-1}+x_d^2 & \text{if $d$ is even}.\end{cases}
    \end{equation*}
    Let $(R,\m,k)$ be an unmixed non-regular local ring of dimension $d$ and characteristic $p$. Then
    \begin{enumerate}
        \item $e_{\text{HK}}(R)\geq e_{\text{HK}}(R_{p,d})$,
        \item if $\eHK{R}=\eHK{R_{p,d}}$, then $\widehat{R}\cong R_{p,d}\widehat{\otimes}_{\mathbb{F}_p} k$, i.e.~the completion of $R$ is isomorphic to the coordinate ring of a quadric hypersurface.
        \item $e_{\text{HK}}(R_{p,d})\geq 1+c_d$,
    \end{enumerate}
    where $c_d$ are the coefficients of the Taylor expansion of $\sec x + \tan x$.
    \end{theorem}
We recall the known formulas for $\eHK{R_{p,d}}$. In characteristic $p\geq 3$, dimensions $4$ through $6$, these where computed in \cite{yos}, and dimension 7 in \cite{ac24}. The values of $\eHK{R_{2,d}}$ are from Theorem \ref{th: HK_function_A_1_char_2}.
\begin{table}[h!]
\centering
\renewcommand{\arraystretch}{1.5}
\setlength{\tabcolsep}{10pt}
\begin{tabular}{c|c|c|c|c|c}
$d$ & 3 & 4 & 5 & 6 & 7\\ \hline
$\eHK{R_{p,d}}$, $p\geq 3$  & $\frac{4}{3}$ & $\frac{29p^2+15}{24p^2+12}$ & $\frac{17p^2+12}{15p^2+10}$ & $\frac{781p^4+656p^2+315}{720p^4+570p^2+270}$ & $\frac{332p^4 + 304p^2 + 192}{315p^4 + 273p^2 + 168}$\\ \hline
$\eHK{R_{2,d}}$  & $\frac{4}{3}$ & $\frac{5}{4} = \frac{30}{24}$ & $\frac{8}{7} = \frac{16}{14}$ & $\frac{9}{8} = \frac{810}{720}$ & $\frac{16}{15} = \frac{336}{315}$
\end{tabular}
\vspace{1em}
\caption{Hilbert–Kunz multiplicity of the quadric hypersurfaces in low dimensions.}
\label{tab: eHK_quadrics_low_dimension}
\end{table}

\vspace{-2em}
\begin{remark}
Note that $\eHK{R_{p,d}}$ is a decreasing function on $p>0$ for $4\leq d\leq 7$. Also, in a recent preprint \cite{pssy25}, it is proven using Ehrhart theory that $\eHK{R_{p,d}}$ can always be written as $f(p^2)/g(p^2)$ where $f$ and $g$ are polynomials of degree $\lfloor d/2\rfloor$
\end{remark}
\begin{lemma}[Proposition 2.13, \cite{wy00}]
    \label{lm: ehkrpd_decreasing_with_dimension}
If $x$ is an $R$-regular element, then $\eHK{R/xR}\geq \eHK{R}$. In particular, for all $p>2$ and $d\geq 1$, we have that $\eHK{R_{p,d}}\geq \eHK{R_{p,d+1}}$.
\end{lemma}

\begin{proof}[Proof of Theorem \ref{th: watanabe_yoshida_low_dimension}]
    We will focus on the first two statements, the strong part of the conjecture, since it has already been proven up to dimension 7 that $\eHK{R_{p,d}}\geq 1+c_d$.

    We prove the result in a series of claims, first solving dimensions 3 and 4, then multiplicity 2 to 4, and finally dimensions 5 and 6.

\begin{claim}
    The Strong Watanabe–Yoshida conjecture holds for dimension $d=3,4$ and characteristic $p>0$:
\end{claim}
    \begin{enumerate}
        \item Let $d = 3$ first. Lemma 3.3 in \cite{wy05} implies that if $e>2$, then $\eHK{R}>\frac{4}{3} = \eHK{R_{2,3}}$. Therefore, we only need to consider the case when $R$ is a double point, i.e., $ e=2$. In that case, there are two options: either $\eHK{R} = 2$, in whose case the result follows, or $\eHK{R}<2$, in whose case $R$ is a hypersurface by Corollary 1.10, \cite{wy05}. In this last case, the result follows from Theorem \ref{maintheoremA}.
        \item Now, let us discuss the four-dimensional case. We will carefully adjust the computations in the proof of Theorem 4.3 \cite{wy05} to obtain sharper bounds that work in the characteristic 2 case.
        \begin{enumerate}
            \item Let $e=2$. The same argument applies; if $\eHK{R}=2$, the result follows since $2>\eHK{R_{2,4}}$, and if $\eHK{R}<2$, then $R$ is hypersurface by Corollary 1.10 \cite{wy05} and the Strong Watanabe–Yoshida conjecture follows from Theorem \ref{maintheoremA}.
            \item Let $e = 3$. In the computation in Claim 1 in the proof of Theorem 4.3, \cite{wy05}, it is shown that $\eHK{R}\geq 3\left(v_s-\frac{2(s-1)^4}{4!}\right)$ for $s\geq 1$. For $s = 2$, this means $\eHK{R}\geq \frac{30}{24} = \frac{5}{4}$. To obtain the strict inequality with the value $\frac{5}{4}$, it is enough by setting $s = 2.1$, and using the formula
            \begin{gather}
                \label{fm: formula_vs}
                v_s = \sum_{n=0}^{\lfloor s\rfloor}(-1)^n\frac{(s-n)^d}{n!(d-n)!}
            \end{gather}
            from Section 3 of \cite{aben12}. This computation yields that $\eHK{R} \geq 1.3 >\frac{5}{4}.$
            \item Let $4\leq e\leq 9$. In the computation in Claim 1 in the proof of Theorem 4.3 in \cite{wy01}, it is shown that $\eHK{R}\geq \frac{(13-e)e}{24}$ for $3\leq e\leq 10$, and $\frac{(13-e)e}{24}>\frac{5}{4}$ whenever $4\leq e\leq 9$.
            \item Let $10\leq e\leq 30$. In Claim 2 in the proof of the same result, it is shown that $\eHK{R}\geq \frac{(78-e)e}{384}$, and $\frac{(78-e)e}{384}>\frac{5}{4}$ for every $10\leq e\leq 30$.
            \item Let $e\geq 31$. In this case, by the usual inequalities between the usual multiplicity and the Hilbert–Kunz multiplicity, we have that $\eHK{R}\geq \frac{e}{d!} \geq \frac{31}{24} >\frac{5}{4}.$
        \end{enumerate}
    \end{enumerate}

\begin{claim}
    \label{cl: low_multiplicity}
The Strong Watanabe–Yoshida conjecture holds for dimension $d\geq 3$, characteristic $p>0$, and multiplicity $e(R) = 2,3,4$.
\end{claim}
    First, observe that we can assume the ring is complete, and therefore the homomorphic image of a Cohen–Macaulay ring.
\begin{enumerate}
    \item If $e(R) = 2$, the Strong Watanabe–Yoshida conjecture follows from Corollary 1.10 in \cite{wy01} and Theorem \ref{maintheoremA}.
    \item As for $e(R) = 3$, if $\eHK{R}\geq \frac{3}{2}$, then $\eHK{R}>\eHK{R_{2,d}}$ for $d\geq 3$. If $\eHK{R}<\frac{3}{2}$, then $R$ is Gorenstein by Theorem 2.4 (iii) in \cite{aben12}, so if $R$ is not a complete intersection, Theorem 4.1 in \cite{aben12} implies that $\eHK{R}\geq 13/8 >\eHK{R_{2,d}}$ for all $d\geq 3$ – see Theorem \ref{th: HK_function_A_1_char_2}. Otherwise, $R$ is a complete intersection, and the result follows from Theorem \ref{maintheoremA}.
    \item Finally, let $e(R) = 4$. If $\eHK{R}\geq \frac{4}{3}$, again $\eHK{R}>\eHK{R_{2,d}}$ for $d\geq 4$. Otherwise, we have that $\eHK{R}<\frac{4}{3}$, and the same argument used in \cite{aben12} implies that $R$ has to be a complete intersection.
\end{enumerate}

\begin{claim}
The Strong Watanabe–Yoshida conjecture holds for dimension $d =5, 6$ and characteristic $p>0$.
\end{claim}
    We revisit the proof of Theorem 5.2 \cite{aben12}. First, we can assume that the residue field of $R$ is infinite, that $\eHK{R}<2$, and therefore that $R$ is a domain. Let $J$ be a minimal reduction of $\m$. In the proof of this theorem, Aberbach and Enescu show that $R$ is either a Cohen–Macaulay ring with minimal multiplicity or $\mu(\mathfrak{m}/J^*)\leq e-2$, where $J^*$ is the tight closure of $J$, thus we study these two cases separately:
    \begin{enumerate}
    \item In the minimal multiplicity case, either $\eHK{R_{p,d}}\geq 1.5$ or $R$ is a hypersurface, so in either case the Strong Watanabe–Yoshida conjecture holds in every positive characteristic by \ref{th: HK_function_A_1_char_2} and \ref{maintheoremA}.
    \item Otherwise, $\mu(\m/J^*)\leq e-2$ and we have to separate between dimensions 5 and 6.
    \begin{enumerate}
        \item Let $d = 5$. The bounds provided in Aberbach–Enescu are sharp for every multiplicity $e\neq 137$, i.e.~$\eHK{R}>\eHK{R_{p,d}}$. When $e = 137$, the lower bound is $e(R)/d! = \frac{137}{5!}<\eHK{R_{p,d}}$, so we need to improve the bound in this case, but simply applying Theorem 3.2 in \cite{aben12} and the expression (\ref{fm: formula_vs}) for $r = 135$ and $s = 1.4$ shows that $\eHK{R}\geq 4.5$.
        \item Let $d = 6$. Again, by the inequality between the Hilbert–Kunz multiplicity and the usual multiplicity, we have that for all $e\geq 811$, 
        \begin{gather*}
            \eHK{R}\geq \frac{811}{720} >\frac{810}{720} = \frac{9}{8}.
        \end{gather*}
        Let us now show the inequality when $5\leq e\leq 810$. For each $s\geq 1$, we define the function
        \begin{gather*}
            G_s(e) := e(v_s-(e-2)v_{s-1}) = -e^2v_s+(v_{s}+2v_{s-1})e.
        \end{gather*}
        We know from Theorem 3.2 in \cite{aben12} that $\eHK{R}\geq G_s(e)$ for every $s\geq 1$. We will prove that for every $5\leq e\leq 810$, $G_s(e)>\frac{9}{8}$.

        For example, if $s = 2.63$ and $5\leq e\leq 9$, it can be easily checked with a computer using (\ref{fm: formula_vs}) that $G_s(e)>\frac{9}{8}$. Table \ref{tab: e_vs_s_dimension_6} shows the value of $s$ that works for each $5\leq e\leq 810$.
        
        \begin{table}[h!]
        \centering
        \renewcommand{\arraystretch}{1.5} 
        \setlength{\tabcolsep}{10pt} 
        \begin{tabular}{c|c}
        $e$ & $s$ \\ \hline
        $5\leq e\leq 9$ & 2.63\\
        $10\leq e\leq 16$ & 2.31 \\
        $17\leq e\leq 28$ & 2.1 \\
        $29\leq e\leq 58$ & 1.9 \\
        $59\leq e\leq 296$ & 1.6 \\
        $297\leq e\leq 810$ & 1.3 \\
        \end{tabular}
        \vspace{1em} 
        \caption{Values of $s$ such that $G_s(e)>\eHK{R_{2,6}}$.}
        \label{tab: e_vs_s_dimension_6}
        \end{table}
    \end{enumerate}
    \end{enumerate}

\end{proof}

\begin{corollary}
    \label{cor: watanabe_yoshida_low_multiplicity}
The Strong Watanabe–Yoshida conjecture holds for dimension $d\geq 2$ and multiplicity $e(R)\leq 5$.
\end{corollary}
\begin{proof}
    The cases of multiplicity $e(R) = 2,3,4$ follow from Claims \ref{cl: low_multiplicity} and Watanabe–Yoshida in dimension 2. Thus, let $e(R) = 5$. Section 3 of \cite{aben12} shows a characteristic free argument that $\eHK{R}\geq 1.112$. By Theorem \ref{th: HK_function_A_1_char_2}, $\eHK{R_{2,7}} = \frac{16}{15}<1.112$. The result follows from the fact that $\eHK{R_{p,7}}$ is a decreasing function (see Table \ref{tab: eHK_quadrics_low_dimension} above), and Lemma \ref{lm: ehkrpd_decreasing_with_dimension}.
\end{proof}

\begin{theorem}
    The Strong Watanabe–Yoshida conjecture holds for dimension $d=7$ and characteristic $p>2$.
\end{theorem}
\begin{proof}
    If the ring is a complete intersection, this is again a consequence of \ref{maintheoremA}. If the ring is not a complete intersection, Aberbach and Cox-Steib provide a series of bounds in Section 4 of \cite{ac24} which are easily seen to be sharp whenever $p>2$ using the fact that $\eHK{R_{p,7}}$ is a decreasing function on $p>2$ and that $\eHK{R_{3,7}} \approx 1.056$.
\end{proof}

\begin{remark}
    Note that $\eHK{R_{2,7}} \approx 1.066$. All the bounds provided by Aberabach and Cox-Steib in dimension 7 are (positive) characteristic-free, and are enough to cover the characteristic 2 case when $e(R)\geq 11$, but not when $6\leq e(R)\leq 10$.
\end{remark}

\bibliographystyle{alpha}
\bibliography{swyci}
\end{document}